\documentclass[11pt]{article}

\usepackage{amssymb,amsmath,amsfonts}
\usepackage{graphicx,color,enumitem}
\usepackage{calrsfs}	
\usepackage{amsthm}	
\usepackage{bm}
\usepackage[round]{natbib}

\usepackage{etoolbox}

\makeatletter

\patchcmd{\NAT@citex}
  {\@citea\NAT@hyper@{%
     \NAT@nmfmt{\NAT@nm}%
     \hyper@natlinkbreak{\NAT@aysep\NAT@spacechar}{\@citeb\@extra@b@citeb}%
     \NAT@date}}
  {\@citea\NAT@nmfmt{\NAT@nm}%
   \NAT@aysep\NAT@spacechar\NAT@hyper@{\NAT@date}}{}{}

\patchcmd{\NAT@citex}
  {\@citea\NAT@hyper@{%
     \NAT@nmfmt{\NAT@nm}%
     \hyper@natlinkbreak{\NAT@spacechar\NAT@@open\if*#1*\else#1\NAT@spacechar\fi}%
       {\@citeb\@extra@b@citeb}%
     \NAT@date}}
  {\@citea\NAT@nmfmt{\NAT@nm}%
   \NAT@spacechar\NAT@@open\if*#1*\else#1\NAT@spacechar\fi\NAT@hyper@{\NAT@date}}
  {}{}

\makeatother

\newcommand{\E}{{\mathbb E}}

\renewcommand{\P}{{\mathbb P}}

\newcommand{\R}{{\mathbb R}}
\renewcommand{\S}{{\mathbb S}}
\newcommand{\N}{{\mathbb N}}

\newcommand{\Acal}{{\mathcal A}}
\newcommand{\Bcal}{{\mathcal B}}
\newcommand{\Ccal}{{\mathcal C}}
\newcommand{\Dcal}{{\mathcal D}}
\newcommand{\Ecal}{{\mathcal E}}

\newcommand{\Gcal}{{\mathcal G}}

\newcommand{\Kcal}{{\mathcal K}}

\newcommand{\Scal}{{\mathcal S}}

\DeclareMathOperator{\rk}{rank}

\newcommand{\Hom}{{\rm Hom}}
\newcommand{\Skew}{{\rm Skew}}

\newcommand{\Pol}{{\rm Pol}}
\newcommand{\Id}{{\mathrm{Id}}}

\DeclareMathOperator{\tr}{Tr}

\DeclareMathOperator{\codim}{codim}

\newtheorem{theorem}{Theorem}

\newtheorem{corollary}[theorem]{Corollary}

\newtheorem{definition}[theorem]{Definition}
\newtheorem{example}[theorem]{Example}

\newtheorem{lemma}[theorem]{Lemma}

\newtheorem{proposition}[theorem]{Proposition}
\newtheorem{remark}[theorem]{Remark}

\numberwithin{equation}{section}
\numberwithin{theorem}{section}

\begin{document}

\title{Polynomial diffusions on compact quadric sets\footnote{The authors would like to thank Guillermo Mantilla-Soler for valuable discussions and suggestions, as well as two anonymous referees for their useful comments. Martin Larsson gratefully acknowledges support from SNF Grant 205121\_163425. The research of Sergio Pulido benefited from the support of the ``Chair Markets in Transition" under the aegis of Louis Bachelier laboratory, a joint initiative of \'Ecole Polytechnique, Universit\'e d'\'Evry Val d'Essonne and F\'ed\'eration Bancaire Fran\c caise, and the project ANR 11-LABX-0019. The research leading to these results has received funding from the European Research Council under the European Union's Seventh Framework Programme (FP/2007-2013) / ERC Grant Agreement n. 307465-POLYTE.}}
\author{Martin Larsson\footnote{ETH Zurich, Department of Mathematics, R\"amistrasse 101, CH-8092, Zurich, Switzerland. Email: martin.larsson@math.ethz.ch} \quad\quad Sergio Pulido\footnote{Laboratoire de Math\'ematiques et Mod\'elisation d'\'Evry (LaMME), Universit\'e d'\'Evry-Val-d'Essonne, ENSIIE, UMR CNRS 8071, IBGBI 23 Boulevard de France, 91037 \'Evry Cedex, France, Email: sergio.pulidonino@ensiie.fr. }}

\maketitle

\begin{abstract}
Polynomial processes are defined by the property that conditional expectations of polynomial functions of the process are again polynomials of the same or lower degree. Many fundamental stochastic processes, including affine processes, are polynomial, and their tractable structure makes them important in applications. In this paper we study polynomial diffusions whose state space is a compact quadric set. Necessary and sufficient conditions for existence, uniqueness, and boundary attainment are given. The existence of a convenient parameterization of the generator is shown to be closely related to the classical problem of expressing nonnegative polynomials---specifically, biquadratic forms vanishing on the diagonal---as a sum of squares. We prove that in dimension $d\le 4$ every such biquadratic form is a sum of squares, while for $d\ge6$ there are counterexamples. The case $d=5$ remains open. An equivalent probabilistic description of the sum of squares property is provided, and we show how it can be used to obtain results on pathwise uniqueness and existence of smooth densities.\\[2ex]

\\[2ex]
\noindent{\textbf {MSC2010 subject classifications:} 60J60, 60H10, 11E25.}
\end{abstract}

\section{Introduction}

Many fundamental stochastic processes appearing in probability theory are {\em polynomial}, meaning that conditional expectations of polynomial functions of the process again have polynomial form; see~\eqref{eq:moments} below. Examples include Brownian motion, geometric Brownian motion, Ornstein-Uhlenbeck processes, squared Bessel processes, Jacobi processes, L\'evy processes, as well as large classes of multidimensional generalizations, such as affine processes and Fisher-Wright diffusions. Polynomial processes have been studied in various degrees of generality by several authors, for instance \citet{Wong:1964}, \citet{Mazet:1997}, \citet{Forman/Sorensen:2008}, \citet{Cuchiero:2011}, \citet{Cuchiero/etal:2012}, \citet{Bakry:2014}, \citet{Bakry/Orevkov/Zani:2014}, \citet{Filipovic/Larsson:2015}. Polynomial processes are also important in applications. In mathematical finance, the polynomial property is useful to build tractable and flexible pricing models in a variety of situations; see for instance \citet{Zhou2003}, \citet{Delbaen/Shirakawa:2002}, \citet{Larsen/Sorensen:2007}, \citet{Gourieroux/Jasiak:2006}, \citet{Cuchiero/etal:2012}, \citet{Filipovic/Gourier/Mancini:2015,Filipovic/Larsson/Trolle:2014}, \citet{Ackerer/Filipovic/Pulido:2016}, and \citet{Filipovic/Larsson:2015}. Every affine diffusion is polynomial; however, non-deterministic affine diffusions do not admit compact state spaces, see~\citet{Kruehner/Larsson:2015}, which may be a drawback in applications. In the present paper we consider polynomial diffusions on (possibly solid) compact quadric sets. This is a natural class of simple state spaces that nonetheless exhibit a rich mathematical structure.

The state space $E$ is defined as follows. Fix $d\in\N$ and let
\[
E = \{x\in\R^d \colon p(x) \ge 0\} \qquad \text{or}\qquad E = \{x\in\R^d \colon p(x) = 0\}
\]
for some $p\in\Pol_2$ such that $E$ is compact and nondegenerate (nonempty and not a point). Here $\Pol_k$ denotes the vector space of polynomial functions on $\R^d$ of total degree less than or equal to~$k$. Up to an affine transformation, $E$ is either the closed unit ball $\Bcal^d$ or the unit sphere $\Scal^{d-1}$; see Section~\ref{S:compact quadrics}.

Next, consider continuous maps $a:\R^d\to\S^d$ and $b:\R^d\to\R^d$ such that
\begin{equation} \label{eq:abpol}
\text{$a_{ij}\in\Pol_2$ and $b_i\in\Pol_1$ for all $i,j\in\{1,\ldots,d\}$},
\end{equation}
where $\S^d$ is the set of $d\times d$ symmetric matrices. We are interested in $E$-valued weak solutions to stochastic differential equations of the form
\begin{equation} \label{eq:SDE}
dX_t = b(X_t)\,dt + \sigma(X_t)\,dW_t,
\end{equation}
where $\sigma:\R^d\to\R^{d\times n}$ for $n\in\N$ is a continuous map with $\sigma\sigma^\top = a$ on $E$, and $W$ is an $n$-dimensional Brownian motion. In particular, we are interested in how probabilistic properties of~\eqref{eq:SDE} are connected to algebraic properties of the coefficients $a$, $b$ and the state space~$E$. 

Diffusions whose coefficients satisfy~\eqref{eq:abpol} are called polynomial. The reason is that~\eqref{eq:abpol} holds if and only if the associated generator $\Gcal$, given by
\begin{equation} \label{eq:G}
\Gcal f(x) = \frac{1}{2}\tr(a(x) \nabla^2 f(x)) + b(x)^\top \nabla f(x),
\end{equation}
preserves polynomials in the sense that $\Gcal\Pol_k\subseteq\Pol_k$ holds for all $k\in\N$. If $X$ is a polynomial diffusion with generator $\Gcal$, a simple argument based on It\^o's formula shows that conditional expectations of polynomials have a particularly simple form: For any $q\in\Pol_k$, one has
\begin{equation}\label{eq:moments}
\E[q(X_T)\mid X_t] = H(X_t)^\top e^{(T-t)G_k}\, \vec q,
\end{equation}
where $H(x)=(h_1(x),\ldots,h_{N_k}(x))$ with $N_k=\dim\Pol_k$ is a basis for $\Pol_k$, and $G_k\in\R^{N_k\times N_k}$ and $\vec q\in\R^{N_k}$ are the corresponding coordinate representations of $\Gcal|_{\Pol_k}$ and $q$, respectively; see \citet[Theorem~3.1]{Filipovic/Larsson:2015} or \citet[Theorem~2.7]{Cuchiero/etal:2012} for details.

Let us summarize the main results of the present paper. Conditions for existence and uniqueness in law of solutions to~\eqref{eq:SDE} can be extracted from known results in the literature, and we briefly review these issues in Section~\ref{S:compact quadrics}. A related question in the case $E=\Bcal^d$ is whether the interior of the state space is stochastically invariant, to which we provide a complete answer; see Proposition~\ref{P:boundary invariant}. Next, in Section~\ref{S:sos} we undertake a more detailed investigation of how to parameterize all weak solutions to~\eqref{eq:SDE}. Theorem~\ref{T:charC} gives a general representation, modulo the requirement that the diffusion matrix $a(x)$ be positive semidefinite on the state space. Characterizing positivity is a difficult problem, which translates into the algebraic question of describing when certain biquadratic forms ${\rm BQ}(x,y)$ are nonnegative. This is closely related to the representability of ${\rm BQ}(x,y)$ as a sum of squares of polynomials. Such a representation always exists if $d\le 4$, but not if $d\ge 6$; see Theorem~\ref{T:sos}. The case $d=5$ remains open. In Section~\ref{S:consequences} we study the probabilistic consequences of the existence of a sum of squares representation. We show that such a representation exists if and only if \eqref{eq:SDE} can be replaced by a certain much more structured SDE; see Theorem~\ref{eq:skBM}. Solutions to this SDE can sometimes be shown to have the pathwise uniqueness property; see Corollary~\ref{C:sphere} and Theorem~\ref{T:pathwise uniqueness}. Moreover, it becomes possible to make detailed assertions regarding the existence of smooth densities by relying on the classical H\"ormander condition; see Theorem~\ref{T:sphere density} and its corollaries. Finally, in Section~\ref{S:algebra}, we collect the algebraic developments needed for the results presented in Section~\ref{S:sos}. Here an important role is played by the so-called Pl\"ucker relations from algebraic geometry.

The following notation will be used throughout this paper: Elements of $\R^d$ are viewed as column vectors. $\Hom_k$ denotes the subspace of $\Pol_k$ consisting of homogeneous polynomials of total degree exactly $k$. For two matrices~$A$ and~$B$ of compatible size, we write $\langle A,B\rangle=\tr(A^\top B)$ for the trace product. The convex cone of positive semi-definite matrices in $\S^d$ is denoted $\S^d_+$, and we often write $A\succeq 0$ to signify that the symmetric matrix $A$ is positive semidefinite. The space of skew-symmetric $d\times d$ matrices is written $\Skew(d)$, and we let $\Skew(2,d)$ denote the subset of rank-two elements. Let $e_i=(0,\ldots,0,1,0,\ldots,0)^\top$ be the $i$th canonical unit vector in $\R^d$. We denote by $D_1,\ldots,D_m$, with $m=\dim\Skew(d)={d\choose 2}$, the elementary skew-symmetric matrices $S_{ij}=e_ie_j^\top - e_je_i^\top$ listed in lexicographic order. That is,
\begin{equation} \label{eq:SijD}
(D_1,\ldots,D_m)=(S_{12},\ldots,S_{1d},S_{23},\ldots,S_{2d},\ldots,S_{d-1,d}).
\end{equation}

\section{Existence, uniqueness, and boundary attainment} \label{S:compact quadrics}

We consider polynomial diffusions whose state space is either the unit ball $\Bcal^d$ or the unit sphere $\Scal^{d-1}$. Up to an affine change of coordinates, this covers all nondegenerate compact quadric sets. Indeed, consider a compact set $E=\{x\in\R^d:p(x)\ge 0\}$ that is nonempty and not a point, where $p\in\Pol_2$. Then $\nabla^2 p$ is constant and negative definite, so by completing the square one sees that $E$ is an affine transformation of $\Bcal^d$. Similarly, the set $\{x\in\R^d\colon p(x)=0\}$ is an affine transformation of~$\Scal^{d-1}$. Note also that the condition~\eqref{eq:abpol} is invariant under affine transformations.

The following sets play an important role in the description of polynomial diffusions on $\Bcal^d$ and $\Scal^{d-1}$:
\begin{equation}\label{eq:Ccal}
\begin{aligned}
\Ccal 	&= \{ c: \R^d\to \S^d : c_{ij} \in \Hom_2 \text{ for all $i,j$, and } c(x)x \equiv 0 \},	\\
\Ccal_+	&= \{ c \in \Ccal : c(x) \in \S^d_+ \text{ for all } x\}.
\end{aligned}
\end{equation}
We have the following crude characterization theorems, most of which follows fairly directly from known results. We start with the unit ball.

\begin{theorem} \label{T:char ball}
Assume $a$ and $b$ satisfy~\eqref{eq:abpol} and the state space is $E=\Bcal^d$. The following conditions are equivalent:
\begin{enumerate}
\item\label{T:char ball:1} There exists a continuous map $\sigma:\R^d\to\R^{d\times d}$ with $\sigma\sigma^\top=a$ on $\Bcal^d$ such that~\eqref{eq:SDE} admits a $\Bcal^d$-valued weak solution for any initial condition in $\Bcal^d$.
\item\label{T:char ball:2} The coefficients $a$ and $b$ are of the form
\begin{equation} \label{T:char ball:eq0}
\begin{aligned}
a(x) &= (1-\|x\|^2)\alpha + c(x), \\
b(x) &= b + Bx,
\end{aligned}
\end{equation}
for some $\alpha\in\S^d_+$, $c\in\Ccal_+$, $b\in\R^d$, and $B\in\R^{d\times d}$ such that
\begin{equation} \label{T:char ball:eq1}
b^\top x + x^\top Bx + \frac{1}{2}\tr(c(x)) \le 0 \quad\text{for all}\quad x\in\Scal^{d-1}.
\end{equation}
\end{enumerate}
\end{theorem}

\begin{proof}
It is proved in~\citet[Proposition~7.1]{Filipovic/Larsson:2015} that \ref{T:char ball:2} is equivalent to the conditions
\begin{equation} \label{T:char ball:eq2}
\text{$a(x)\succeq 0$ on $E$}, \quad \text{$a(x)x=0$ on $\Scal^{d-1}$}, \quad \text{$\Gcal p(x)\ge0$ on $\Scal^{d-1}$,}
\end{equation}
where $p(x)=1-\|x\|^2$ and $\Gcal$ is given by~\eqref{eq:G}. That \ref{T:char ball:1} implies \eqref{T:char ball:eq2} follows from Theorem~6.1 in \citet{Filipovic/Larsson:2015}. The reverse implication follows from Theorem~6.3 and Proposition~7.1 in \citet{Filipovic/Larsson:2015} when the inequality in~\eqref{T:char ball:eq1} is strict. For the general case one can, for instance, invoke Example~2.5 and Theorem~2.2 in \citet{DaPrato/Frankowska:2007}. Alternatively, one can use~\eqref{T:char ball:eq2} to verify the positive maximum principle directly, and then apply \citet[Theorem~4.5.4]{Ethier/Kurtz:2005}.
\end{proof}

If $d=1$, the polynomial diffusions on $\Bcal^1=[-1,1]$ are simply the Jacobi processes $dX_t=(b+BX_t)dt+\sigma\sqrt{1-X_t^2}dW_t$. In this case $\Ccal_+=\Ccal=\{0\}$. In contrast, for $d>1$ the set $\Ccal_+$ is non-trivial, and its elements $c$ induce diffusive fluctuations tangentially to the sphere due to the property $c(x)x\equiv0$. For example, our results imply that for $d=2$ one can always represent polynomial diffusions on the unit ball as solutions to the SDE
\[
dX_t = (b+BX_t) dt + \sqrt{1-X_{1t}^2-X_{2t}^2}\, \sigma \begin{pmatrix} dW^1_t \\ dW^2_t \end{pmatrix} + \nu \begin{pmatrix} X_{2t} \\ - X_{1t} \end{pmatrix} dW^{3}_t
\]
for some $b\in\R^2$, $B\in\R^{2\times 2}$, $\sigma\in\R^{2\times 2}$, and $\nu\in\R$. Here the presence of the term involving the third Brownian motion $W^3$ gives rise to tangential diffusion. In higher dimensions, explicit parameterizations become more difficult to obtain. Indeed, Theorem~\ref{T:char ball} provides no explicit description of~$\Ccal_+$, and a significant part of the present paper is devoted to studying this set in detail. This is the topic Section~\ref{S:sos}.

It is frequently of interest to know whether a solution~$X$ to~\eqref{eq:SDE} that starts in the interior of $\Bcal^d$ remains in the interior. This question of boundary attainment is resolved by the following result, which is a refinement of Theorem~5.7 in \citet{Filipovic/Larsson:2015}. In its statement, we assume that $a(x)$ and $b(x)$ satisfy~\eqref{eq:abpol} along with the two equivalent conditions in Theorem~\ref{T:char ball}. In particular, \eqref{T:char ball:eq1} is satisfied. Let $\P_x$ denote the law of the solution $X$ to~\eqref{eq:SDE} with initial condition $x\in\Bcal^d$. A measurable subset $D\subseteq \Bcal^d$ is said to be {\em invariant for $X$} if, for every $x\in D$, one has $\P_x(X_t\in D\text{ for all }t\ge0)=1$.

\begin{proposition} \label{P:boundary invariant}
The interior $\Bcal^d\setminus\Scal^{d-1}$ is invariant for $X$ if and only if
\begin{equation} \label{bdry attain}
b^\top x + x^\top (B+\alpha)x + \frac{1}{2}\tr(c(x)) \le 0 \quad\text{for all}\quad x\in\Scal^{d-1}.
\end{equation}
\end{proposition}

\begin{proof}
Let $p(x)=1-\|x\|^2$ and note that
\[
a(x)\nabla p(x) = p(x)h(x) \quad \text{where} \quad  h(x)=-2\alpha x.
\]
Furthermore,
\begin{equation}\label{eq:att_ref0}
2\,\Gcal p(x) - h(x)^\top \nabla p(x) = -2p(x)\tr(\alpha) - 4\left(b^\top x + x^\top (B+\alpha)x + \frac{1}{2}\tr(c(x))\right).
\end{equation}
If \eqref{bdry attain} fails, then the right-hand side of~\eqref{eq:att_ref0} is strictly negative for some $x\in\Scal^{d-1}$. Also, $\Gcal p(x)\ge0$ due to~\eqref{T:char ball:eq1}. Thus \citet[Theorem~5.7(iii)]{Filipovic/Larsson:2015} implies that $\Bcal^d\setminus\Scal^{d-1}$ is not invariant for $X$ (a look at the proof of that result shows that its assumption that $\{t:X_t\in\Scal^{d-1}\}$ has Lebesgue measure zero is not needed when $X_0\notin\Scal^{d-1}$).

Suppose now~\eqref{bdry attain} holds, and let $X_0=x_0\in\Bcal^d\setminus\Scal^{d-1}$. Define $\tau=\inf\{t: X_t\in\Scal^{d-1}\}$. We must show that $\P_{x_0}(\tau<\infty)=0$. For $t\in[0,\tau)$, It\^o's formula yields
\[
\log p(X_t) = \log p(x_0) + \int_0^t \frac{2\,\Gcal p(X_s) - h(X_s)^\top \nabla p(X_s)}{2p(X_s)}\,ds + \int_0^t \frac{\nabla p(X_s)^\top \sigma(X_s)}{p(X_s)}\,dW_s.
\]
Suppose we can find a constant $\kappa_1>0$ such that
\begin{equation} \label{eq:bdry attain 1}
2\,\Gcal p(x) - h(x)^\top \nabla p(x) \ge -2\kappa_1 p(x)\quad\text{for all}\quad x\in\Bcal^d.
\end{equation}
Then $\log p(X_t) + \kappa_1t$ is a local submartingale on $[0,\tau)$, bounded from above on bounded time intervals. The submartingale convergence theorem then implies that $\lim_{t\to T\wedge\tau}\log p(X_t)$ exists in $\R$ for any deterministic $T<\infty$, so that $T<\tau$. We deduce that $\tau=\infty$, as desired. This is the same ``McKean's argument'' as in the proof of \citet[Theorem~5.7(i)--(ii)]{Filipovic/Larsson:2015}. A suitable version of the submartingale convergence theorem can found in, for instance, \citet[Lemma~4.14]{Larsson/Ruf:2014}.

It remains to argue~\eqref{eq:bdry attain 1}. Since $\tr(c(x))$ is a quadratic form in $x$, we can find $C\in\S^d$ such that $\tr(c(x))=x^\top Cx$. Define $\Sigma=-\alpha-\frac{1}{2}(B+B^\top + C)\in\S^d$. In view of~\eqref{eq:att_ref0}, condition~\eqref{eq:bdry attain 1} then becomes
\begin{equation} \label{eq:bdry attain 11}
x^\top \Sigma x - b^\top x \ge -\kappa (1-\|x\|^2) \quad\text{for all}\quad x\in\Bcal^d,
\end{equation}
where the constant $\kappa$ is related to $\kappa_1$ by $2\kappa=\kappa_1-\tr(\alpha)$. On the other hand, the assumption~\eqref{bdry attain} written in terms of $\Sigma$ becomes
\begin{equation} \label{eq:att_ref1}
x^\top\Sigma x - b^\top x \ge 0 \quad\text{for all}\quad x\in\Scal^{d-1}.
\end{equation}
We claim that~\eqref{eq:att_ref1} in fact implies that~\eqref{eq:bdry attain 11} holds for $\kappa=\|b\|$. For $x=0$ the statement is obvious. For $x\ne0$, define $\overline x=x/\|x\|$ and apply~\eqref{eq:att_ref1} to get $\overline x^\top \Sigma \overline x \ge b^\top \overline x$. Multiplying both sides by $\|x\|^2$, rearranging, and applying the Cauchy-Schwartz inequality gives
\[
x^\top \Sigma x - b^\top x  \ge -b^\top x (1-\|x\|)  \ge -\|b\| \|x\| (1-\|x\|) \ge -\|b\| (1-\|x\|^2),
\]
as required.
\end{proof}

The case of the sphere is simpler. In particular, the question of boundary attainment becomes vacuous.

\begin{theorem} \label{T:char sphere}
Assume $a$ and $b$ satisfy~\eqref{eq:abpol} and the state space is $E=\Scal^{d-1}$. The following conditions are equivalent:
\begin{enumerate}
\item\label{T:char sphere:1} There exists a continuous map $\sigma:\R^d\to\R^{d\times d}$ with $\sigma\sigma^\top=a$ on $\Scal^{d-1}$ such that~\eqref{eq:SDE} admits an $\Scal^{d-1}$-valued weak solution for any initial condition in $\Scal^{d-1}$.
\item\label{T:char sphere:2} The coefficients $a$ and $b$ satisfy $a = c$ on $\Scal^{d-1}$ for some $c \in \Ccal_+$, and $b(x) = Bx$ for some $B\in\R^{d\times d}$, where
\begin{equation} \label{eq:C:char sphere}
2 x^\top B x + \tr(c(x)) \equiv 0.
\end{equation}
\end{enumerate}
\end{theorem}

\begin{proof}
\citet[Theorem~6.3]{Filipovic/Larsson:2015} yields that~\ref{T:char sphere:2} implies~\ref{T:char sphere:1}. For the reverse implication, one argues as in the proof of Theorem~\ref{T:char ball} that~\ref{T:char sphere:1} implies Theorem~\ref{T:char ball}\ref{T:char ball:2}, but with equality in~\eqref{T:char ball:eq1}. This in turn implies that $a(x)=c(x)$ for $x\in\Scal^{d-1}$ and that $b=0$, whence~\eqref{eq:C:char sphere} holds.
\end{proof}

Theorems~\ref{T:char ball} and~\ref{T:char sphere} do not make any uniqueness statements. However, since the state space~$E$ is compact, the polynomial property implies that uniqueness in law for~\eqref{eq:SDE} always holds. More precisely, we have the following result:

\begin{lemma} \label{L:law}
Assume $a$ and $b$ satisfy~\eqref{eq:abpol}, let $\sigma:\R^d\to\R^{d\times n}$, $n\in\N$, be a continuous map with $\sigma\sigma^\top=a$ on $E$, and let $x\in E$. Let $X$ be an $E$-valued weak solution to~\eqref{eq:SDE} with initial condition $X_0=x$. Then the law of $X$ is uniquely determined by $a$, $b$, and $x$.
\end{lemma}

\begin{proof}
See \citet[Corollary~5.2]{Filipovic/Larsson:2015}. The argument relies on the fact that $a$, $b$, and $x$ uniquely determine all joint moments of all finite-dimensional marginal distributions of $X$ due to~\eqref{eq:moments}; see \citet[Corollary~3.2]{Filipovic/Larsson:2015} for details. Due to the compactness of $E$, this in turn determines the law of $X$.
\end{proof}

\begin{remark} \label{R:pwunique}
So far no assertions have been made regarding pathwise uniqueness of solutions to~\eqref{eq:SDE}. Pathwise uniqueness depends on the choice of $\sigma$, and the existence of a suitable $\sigma$ is a delicate question in general. See \citet[Remark~2.3]{Spreij/Veerman:2012} for an instructive discussion in connection with affine diffusions on parabolic sets. Sometimes pathwise uniqueness is immediate: If $a(x)$ is positive definite on the interior $\Bcal^d\setminus\Scal^{d-1}$ then $\sigma(x)=a(x)^{1/2}$ is locally Lipschitz there, which by standard arguments yields pathwise uniqueness up to the first hitting time of the boundary. In Section~\ref{S:pathwise} we discuss other situations where pathwise uniqueness can be established.
\end{remark}

\section{Tangential diffusion, biquadratic forms, and sums of squares} \label{S:sos}

In this section we undertake a detailed analysis of the linear space $\Ccal$ and the convex cone $\Ccal_+$ appearing in Theorems~\ref{T:char ball} and~\ref{T:char sphere}, and whose role is to generate diffusive movements tangentially to $\Scal^{d-1}$. It tuns out that a complete description is difficult to obtain in general; see Theorem~\ref{T:sos} below. Fortunately, a partial characterization suitable for applications is within reach. We start with some general remarks, introducing some concepts and notation along the way.

Any map $c:\R^d\to\S^d$ whose components lie in $\Hom_2$ induces a map ${\rm BQ}:\R^d\times\R^d\to\R$ via the expression
\begin{equation} \label{eq:BQcx}
{\rm BQ}(x,y) = y^\top c(x) y.
\end{equation}
This map is a {\em biquadratic form}: a quadratic form in $x$ for each fixed $y$, and a quadratic form in $y$ for each fixed $x$. If $c\in\Ccal_+$, then additionally ${\rm BQ}(x,y)\ge 0$ and ${\rm BQ}(x,x)= 0$ for all $(x,y)$. The following lemma states that the converse is true as well.

\begin{lemma} \label{L:BQC+}
The set $\Ccal_+$ is in one-to-one correspondence with the set of all nonnegative biquadratic forms vanishing on the diagonal. The correspondence is given by~\eqref{eq:BQcx}.
\end{lemma}

\begin{proof}
Any biquadratic form ${\rm BQ}(x,y)$ in $d+d$ variables induces a map $c:\R^d\to\S^d$ with components in $\Hom_2$ by the formula $c_{ij}(x)=\partial^2_{y_iy_j}{\rm BQ}(x,y)$ for $i\ne j$, and $c_{ii}(x)=\partial^2_{y_iy_i}{\rm BQ}(x,y)/2$. Furthermore, if ${\rm BQ}(x,y)\ge 0$ for all $(x,y)$, then $c(x)\succeq 0$ for all $x$. In this case, ${\rm BQ}(x,x)=0$ implies $\|c(x)^{1/2}x\|^2=x^\top c(x)x=0$, and hence $c(x)x=0$, for all~$x$. Thus $c\in\Ccal_+$.
\end{proof}

\begin{remark}
Note that the set $\Ccal$ {\em does not} correspond one-to-one with the set of all (not necessarily nonnegative) biquadratic forms vanishing on the diagonal. The reason is that $x^\top c(x) x\equiv 0$ does not imply $c(x)x\equiv0$ in the absence of positive semidefiniteness. An example is the map
\[
c(x) = \begin{pmatrix} -2x_1x_2 & x_1^2 \\ x_1^2 & 0 \end{pmatrix},
\]
which satisfies $x^\top c(x)x\equiv 0$ but not $c(x)x\equiv 0$. It will become apparent that the relevant stepping stone toward an understanding of $\Ccal_+$ is the set $\Ccal$, not the set of biquadratic forms vanishing on the diagonal.
\end{remark}

Our first main result is a characterization of $\Ccal$. Let $m={d \choose 2}=\dim\Skew(d)$. For any $H=(h_{pq})\in\S^m$, define a map $c_H$ by
\begin{equation} \label{eq:cH(x)}
c_H(x) = \sum_{p,q=1}^m h_{pq} \,D_p\,xx^\top D_q^\top,
\end{equation}
where $D_1,\ldots,D_m$ are given by \eqref{eq:SijD}. It is clear that $c_H\in\Ccal$, and it turns out that every element of $\Ccal$ is in fact of this form. Naively one might then conjecture that the dimension of $\Ccal$ equals $\dim\S^m={m+1\choose 2}$, the number of free parameters $h_{pq}$ appearing in~\eqref{eq:cH(x)}. However, this turns out to be incorrect. Indeed, there exist linear relations among the maps $D_p\,xx^\top D_q^\top$ in~\eqref{eq:cH(x)}; to capture them, we consider the linear space
\begin{equation} \label{eq:DK}
\Kcal = \left\{ H \in \S^m : c_H(x)\equiv 0 \right\}.
\end{equation}
The proof of the following theorem is provided in Section~\ref{S:proof_charC}.

\begin{theorem} \label{T:charC}
The set $\Ccal$ and its dimension are given by
\begin{equation} \label{T:charC:1}
\Ccal = \left\{ c_H : H \in\S^m\right\}
\qquad\text{and}\qquad
\dim \Ccal = 2{d\choose 4} + 3{d\choose 3} + {d\choose 2} = \frac{d^2(d^2-1)}{12},
\end{equation}
where $m={d\choose 2}$ and we set ${d \choose k}=0$ for $k>d$. Moreover, $\dim \Kcal = \dim\S^m-\dim\Ccal= {d\choose 4}$.
\end{theorem}

The linear space $\Kcal$ has several interesting properties that are used in the proofs of our subsequent results. This is discussed in detail in Section~\ref{S:setK}. In a nutshell, the situation is the following: Each $H\in\S^m$ naturally corresponds to a quadratic form on the space of skew-symmetric matrices, which when evaluated at $A=xy^\top-yx^\top\in\Skew(2,d)$ yields precisely ${\rm BQ}(x,y)$; see~\eqref{eq:yCHxy}. Now, it turns out that $\Kcal$ admits a set of basis vectors that can be identified with the so-called Pl\"ucker polynomials from algebraic geometry. The zero set of these polynomials is isomorphic to $\Skew(2,d)$. As a consequence, the quotient $\S^m/\Kcal$ can be identified with the set of quadratic forms acting on $\Skew(2,d)$. Thus, in view of Theorem~\ref{T:charC}, each element of $\Ccal$ corresponds to a unique quadratic form on $\Skew(2,d)$. Finally, we mention that it is possible to explicitly write down a basis for $\Kcal$; see Remark~\ref{R:Kbasis}.

We now turn to the delicate question of positivity: $c(x)\succeq 0$ for all $x$, or equivalently ${\rm BQ}(x,y)\ge 0$ for all $(x,y)$. A simple example where this holds is $c(x)=A\,xx^\top A^\top$, where $A\in\Skew(d)$. More generally, by taking conic combinations, elements of the form
\begin{equation} \label{eq:cxA1Am}
c(x) = A_1\,xx^\top A_1^\top + \cdots + A_m\,xx^\top A_m^\top, \qquad A_1,\ldots,A_m\in\Skew(d),
\end{equation}
lie in $\Ccal_+$. The corresponding biquadratic form is given by
\[
{\rm BQ}(x,y)= (y^\top A_1 x)^2 + \cdots + (y^\top A_m x)^2,
\]
which has the important property that it can be written as a sum of squares of polynomials (indeed, of quadratic forms). The converse is also true: whenever ${\rm BQ}(x,y)$ is a sum of squares, the corresponding $c(x)$ is of the form~\eqref{eq:cxA1Am}. Specifically, one has the following:

\begin{lemma} \label{L:sos}
Let $c\in\Ccal$, $m={d\choose 2}$. The following conditions are equivalent.
\begin{enumerate}
\item\label{L:sos:1} ${\rm BQ}(x,y)=y^\top c(x)y$ is a sum of squares of polynomials.
\item\label{L:sos:2} $c=c_H$ for some $H\in\S^m_+$.
\item\label{L:sos:3} $c(x) = \sum_{p=1}^m A_p\,xx^\top A_p^\top$ for some $A_1,\ldots,A_m\in\Skew(d)$.
\end{enumerate}
\end{lemma}

\begin{proof}
\ref{L:sos:3} $\Longrightarrow$ \ref{L:sos:1}: Obvious.

\ref{L:sos:1} $\Longrightarrow$ \ref{L:sos:2}: We first show that \ref{L:sos:1} implies that there exist $m'\in\N$ and $A_1,\ldots,A_{m'}\in\Skew(d)$ such that $y^\top c(x)y=\sum_{p=1}^{m'} (y^\top A_p\,x)^2$. To this end, suppose $y^\top c(x) y=f_1(x,y)^2+\cdots+f_{m'}(x,y)^2$ for some polynomials $f_p$, $p=1,\ldots,m'$, where $m'\in\N$. By setting $x=y=0$, we see that $f_p(0,0)=0$ for all $p$. Since also $f_p(sx,y)^2\le y^\top c(sx)y=s^2 y^\top c(x)y$, each $f_p$ is linear in $x$. Similarly, each $f_p$ is linear in $y$. It follows that $f_p(x,y)=y^\top A_p x$ for some matrix $A_p\in\R^{d\times d}$. Finally, $(x^\top A_p x)^2 \le x^\top c(x)x\equiv 0$, whence $A_p$ is skew-symmetric. This proves the claimed representation.

Now, since $D_1,\ldots,D_m$ form a basis for $\Skew(d)$, for each $p$ there exists $u_p=(u_p^1,\ldots,u_p^m)\in\R^m$ such that $A_p=\sum_{i=1}^m u_p^i\, D_i$. Thus,
\[
y^\top c(x) y = y^\top \Big( \sum_{p=1}^{m'} A_p \, xx^\top  A_p^\top \Big) y = y^\top \Big( \sum_{i,j=1}^m h_{ij} \,D_i\,xx^\top D_j^\top \Big) y,
\]
where $H=(h_{ij})\in\S^m$ is given by $H=u_1u_1^\top +\cdots + u_{m'}u_{m'}^\top$ and thus is positive semidefinite. We deduce~\ref{L:sos:2}.

\ref{L:sos:2} $\Longrightarrow$ \ref{L:sos:3}: Writing $H=u_1u_1^\top+\cdots+u_mu_m^\top$ for some vectors $u_p$ with components $(u_p^1,\ldots,u_p^m)$ and substituting into~\eqref{eq:cH(x)} yields
\[
c_H(x) = \sum_{p=1}^m (u^1_pD_1+\cdots+u^m_pD_m)\, x x^\top (u^1_pD_1+\cdots+u^m_pD_m)^\top,
\]
which is of the desired form with $A_p=\sum_{i=1}^m u_p^i\, D_i$.
\end{proof}

Lemma~\ref{L:sos} shows that those $c\in\Ccal_+$ whose associated biquadratic form is a sum of squares have a nice parameterization in terms of positive semidefinite matrices~$H$. One is thus naturally led to ask whether every nonnegative biquadratic form vanishing on the diagonal can be written as a sum of squares. Our next main result addresses this question. Its proof is provided in Sections~\ref{S:proof_SOS1} and~\ref{S:T:sos:2}, relying on the material developed in Section~\ref{S:setK}. The proof also depends on a known result (Lemma~\ref{L:rank1}) on existence of low-rank elements in the intersection of~$\S^m_+$ with certain affine subspaces of~$\S^m$.

\begin{theorem} \label{T:sos}
\begin{enumerate}
\item\label{T:sos:1} If $d\le 4$, then any nonnegative biquadratic form in $d+d$ variables vanishing on the diagonal is a sum of squares. Equivalently, any $c\in\Ccal_+$ is of the form $c=c_H$ for some $H\in\S^m_+$.
\item\label{T:sos:2} If $d\ge 6$, then there exists a nonnegative biquadratic form in $d+d$ variables vanishing on the diagonal that is not a sum of squares. Equivalently, there exists $c\in\Ccal_+$ that is not of the form $c=c_H$ for some $H\in\S^m_+$.
\end{enumerate}
\end{theorem}

\begin{remark}
In \citet{laszlo2010sum} it was observed that the conclusion of Theorem~\ref{T:sos}\ref{T:sos:1} holds for $d\le3$, while the case $d\in\{4,5\}$ was left open; see also \citet{laszlo2012sum}. The case $d=3$ can also be deduced from \citet{quarez2015real}. We have not been able to determine whether there exists a nonnegative biquadratic form in $5+5$ variables, vanishing on the diagonal, that is not a sum of squares. The counterexample that establishes Theorem~\ref{T:sos}\ref{T:sos:2} comes from~\citet{laszlo2010sum}. For the benefit of the reader we recap this construction in Section~\ref{S:T:sos:2} using the setting and notation of the present paper. The study of sum of squares representations of nonnegative biquadratic forms (not necessarily vanishing on the diagonal) goes back to \citet{choi1975positive}.
\end{remark}

\section{Consequences of the sum of squares property} \label{S:consequences}

Throughout this section, we assume that $a(x)$ and $b(x)$ satisfy \eqref{T:char ball:eq0}--\eqref{T:char ball:eq1} for some fixed $\alpha\in\S^d_+$, $c\in\Ccal_+$, $b\in\R^d$, $B\in\R^{d\times d}$, and let ${\rm BQ}(x,y)=y^\top c(x) y$ be the associated nonnegative biquadratic form. Theorem~\ref{T:char ball} guarantees that a $\Bcal^d$-valued weak solution to~\eqref{eq:SDE} exists for a suitable choice of $\sigma$ and for any initial condition $x\in\Bcal^d$. Denote its law by $\P_x$, and observe that it is uniquely determined by $a$, $b$, and $x$; see Lemma~\ref{L:law}. In this section we discuss consequences of the existence of a sum of squares representation of ${\rm BQ}(x,y)$.

\subsection{SDE representation}

Our first result shows that $\P_x$ can be represented as the law of the solution to a certain structured SDE if and only if ${\rm BQ}(x,y)$ is a sum of squares. Before stating it, we give the following definition:

\begin{definition}
A process $Y$ is called a {\em $\Skew(d)$-valued correlated Brownian motion with drift} if it is of the form
\begin{equation} \label{eq:skBM}
Y_t = A_0\, t + A_1\widehat W^1_t + \cdots + A_m\widehat W^m_t,
\end{equation}
where $A_0,\ldots,A_m\in\Skew(d)$ and $\widehat W=(\widehat W^1,\ldots\widehat W^m)$ is an $m$-dimensional Brownian motion.
\end{definition}

\begin{theorem} \label{T:SDE}
The following conditions are equivalent:
\begin{enumerate}
\item\label{T:SDE:1} ${\rm BQ}(x,y)=y^\top c(x) y$ is a sum of squares of polynomials.
\item\label{T:SDE:2} For each $x\in\Bcal^d$, $\P_x$ is the law of the unique $\Bcal^d$-valued weak solution to the SDE
\begin{equation} \label{eq:SDE:odYX}
dX_t = (b + \widehat B X_t)\,dt + \sqrt{1-\|X_t\|^2}\,\alpha^{1/2}\,dW_t + (\circ\,dY_t) X_t
\end{equation}
with initial condition $X_0=x$, where $-\widehat B\in\S^d_+$, $W=(W^1,\ldots,W^d)$ is a $d$-dimensional Brownian motion, and $Y$ is a $\Skew(d)$-valued correlated Brownian motion with drift, independent of $W$. Here $\circ\,dY_t$ denotes Stratonovich differential.
\end{enumerate}
In this case, $B$ is related to $\widehat B$ and $A_0,\ldots,A_m$ in~\eqref{eq:skBM} through
\begin{equation}\label{eq:T:SDE:1}
\frac{1}{2}(B-B^\top) = A_0\qquad\text{and} \qquad \frac{1}{2}(B+B^\top) = \widehat B - \frac{1}{2}\sum_{p=1}^m A_p^\top A_p.
\end{equation}
\end{theorem}

\begin{proof}
\ref{T:SDE:1} $\Longrightarrow$ \ref{T:SDE:2}: Lemma~\ref{L:sos} implies that
\begin{equation} \label{eq:T:SDE:pf0}
c(x) = \sum_{p=1}^m A_p\, xx^\top A_p^\top
\end{equation}
for some $A_1,\ldots,A_m\in\Skew(d)$. Define $A_0\in\Skew(d)$ and $\widehat B\in \S^d$ via~\eqref{eq:T:SDE:1}. We claim that $-\widehat B\in\S^d_+$. Indeed, writing the condition~\eqref{T:char ball:eq1} in terms of $\widehat B$, and using that $x^\top A_0x=0$ by skew-symmetry, yields
\begin{equation} \label{eq:T:SDE:pf1}
b^\top x + x^\top \widehat Bx \le 0 \quad\text{for all}\quad x\in\Scal^{d-1}.
\end{equation}
Inserting $x$ and $-x$ into this inequality and adding, one obtains $x^\top \widehat Bx \le 0$ for all $x\in\Scal^{d-1}$, whence $-\widehat B\in\S^d_+$ as claimed. Next, writing the SDE~\eqref{eq:SDE:odYX} in It\^o form yields
\[
dX_t = (b+BX_t)\,dt + \widehat\sigma(X_t)\begin{pmatrix}dW_t\\ d\widehat W_t\end{pmatrix},
\]
where
\begin{equation}\label{eq:sigma hat}
\widehat\sigma(x) = \Big(\begin{matrix} \sqrt{1-\|x\|^2}\,\alpha^{1/2} & A_1x &\cdots& A_mx \end{matrix}\Big) \in \R^{d\times (d+m)}.
\end{equation}
In view of~\eqref{eq:T:SDE:pf0} we have $\widehat\sigma(x)\widehat\sigma(x)^\top = (1-\|x\|^2)\alpha + c(x) = a(x)$. Thus~\eqref{eq:SDE:odYX} has the same drift and diffusion coefficients as~\eqref{eq:SDE}, which by Lemma~\ref{L:law} implies that the law of any solution to~\eqref{eq:SDE:odYX} with initial condition~$x$ is indeed~$\P_x$, as desired. Finally, the existence of a $\Bcal^d$-valued solution to~\eqref{eq:SDE:odYX} follows, as in the proof of Theorem~\ref{T:char ball}, from Example~2.5 and Theorem~2.2 in \citet{DaPrato/Frankowska:2007}.

\ref{T:SDE:2} $\Longrightarrow$ \ref{T:SDE:1}: Let $\widehat\sigma(x)$ be given by~\eqref{eq:sigma hat}. Since $\P_x$ is the uniquely determined law of the solution to~\eqref{eq:SDE:odYX}, one obtains
\[
a(x) = \widehat\sigma(x)\widehat\sigma(x)^\top = (1-\|x\|^2)\alpha + \sum_{p=1}^m A_p\, xx^\top A_p^\top.
\]
Since $a(x)=(1-\|x\|^2)\alpha + c(x)$, it follows that~\eqref{eq:T:SDE:pf0} holds. Lemma~\ref{L:sos} then implies \ref{T:SDE:1}.
\end{proof}

The terms on the right-hand side of~\eqref{eq:SDE:odYX} have natural interpretations. Since $\widehat B$ is symmetric with real non-positive eigenvalues, the drift has a mean-reverting component $\widehat BX_t$ in addition to the constant part $b$. The term involving $W_t$ induces diffusive motion along the eigenvectors of $\alpha$, with quadratic variation that is scaled by the squared distance $1-\|X_t\|^2$ to the boundary of~$\Bcal^d$. Finally, the term $(\circ\,dY_t) X_t$ can be viewed as an infinitesimal random rotation of $X_t$; see also \citet{price1983rolling} and \citet{van1985brownian}. Heuristically, the Stratonovich increment $\circ \,dY_t$ is a small random Skew-symmetric matrix that maps $X_t$ to an infinitesimal tangent vector of the ball with radius $\|X_t\|$. We prefer using the Stratonovich form here, as it allows us to clearly separate tangential motion from radial motion in~\eqref{eq:SDE:odYX}. If one were to switch to It\^o form, the resulting drift correction would interfere with this interpretation. In particular, even for specifications where $b=0$, $\widehat B=\alpha=0$, for which $\|X\|$ is constant, an inward-pointing ``drift'' would appear in the It\^o formulation. See also Corollary~\ref{C:sphere} below.

\begin{remark}
In the setting of Theorem~\ref{T:SDE}, \eqref{eq:T:SDE:1} implies that the necessary and sufficient condition~\eqref{bdry attain} for boundary attainment reduces to
\[
b^\top x + x^\top (\widehat B + \alpha)x \le 0 \quad\text{for all}\quad x\in\Scal^{d-1}.
\]
\end{remark}

If $X$ takes values in $\Scal^{d-1}$, the sum of squares property of ${\rm BQ}(x,y)$ has stronger implications. In particular, $X$ can now be constructed as the {\em unique strong} solution to a suitable SDE.

\begin{corollary} \label{C:sphere}
Assume the coefficients $a$ and $b$ satisfy Theorem~\ref{T:char sphere}\ref{T:char sphere:2}. The following conditions are equivalent:
\begin{enumerate}
\item\label{C:sphere:1} ${\rm BQ}(x,y)=y^\top c(x) y$ is a sum of squares of polynomials.
\item\label{C:sphere:2} For each $x\in\Scal^{d-1}$, $\P_x$ is the law of the $\Scal^{d-1}$-valued unique strong solution to the SDE
\begin{equation} \label{eq:odYX}
dX_t = (\circ\,dY_t) X_t
\end{equation}
with initial condition $X_0=x$, where $Y$ is a $\Skew(d)$-valued correlated Brownian motion with drift.
\item\label{C:sphere:3} The operator $\Gcal=\frac{1}{2}\tr(a\,\nabla^2)+b^\top \nabla$ corresponding to the SDE~\eqref{eq:SDE} can be expressed in H\"ormander form as
\begin{equation}\label{eq:Horm}
\Gcal = V_0 + \frac{1}{2} \sum_{p=1}^m V_p^2,
\end{equation}
where for $p=0,\ldots,m$, $V_p$ is the linear vector field given by $V_p(x)=A_px$ for some $A_p\in\Skew(d)$.
\end{enumerate}
\end{corollary}

\begin{proof}
\ref{C:sphere:1} $\Longleftrightarrow$ \ref{C:sphere:2} is deduced from the corresponding equivalence in Theorem~\ref{T:SDE}. Indeed, the forward implication follows since $b=0$, $\alpha=0$, and $\widehat B=0$. Here the latter is due to~\eqref{eq:C:char sphere} and~\eqref{eq:T:SDE:1}, which give
\[
x^\top \widehat B x = x^\top B x + \frac{1}{2}\sum_{p=1}^m x^\top A_pA_p^\top x = x^\top B x + \frac{1}{2}\tr(c(x))=0, \qquad x\in\R^d.
\]
Note that the form of~$c(x)$ follows from Lemma~\ref{L:sos}. For the converse, simply note that \ref{C:sphere:2} implies that Theorem~\ref{T:SDE}\ref{T:SDE:2} is satisfied with $b=0$, $\alpha=0$, and $\widehat B=0$. This yields~\ref{C:sphere:1}. The implication \ref{C:sphere:2} $\Longrightarrow$ \ref{C:sphere:3} is well-known, and follows from a brief calculation using the identity $V^2_pf(x)=\tr(\nabla^2 f(x)A_p\,xx^\top A_p^\top) - (A_p^\top A_p\, x)^\top \nabla f(x)$. Finally, the implication \ref{C:sphere:3} $\Longrightarrow$ \ref{C:sphere:1} is again a calculation showing that $c(x)=A_1\,xx^\top A_1^\top+\cdots+A_m\,xx^\top A_m^\top$, which gives~\ref{C:sphere:1} via Lemma~\ref{L:sos}.
\end{proof}

\begin{example}
The classical Brownian motion on the sphere is covered by Corollary~\ref{C:sphere}. In particular, it is a polynomial diffusion. Indeed, by taking $A_p=D_p$ for $p=1,\ldots,m$, one obtains
\[
c(x) = \sum_{i<j} (e_ie_j^\top-e_je_i^\top) xx^\top(e_je_i^\top-e_ie_j^\top) = (x^\top x)\Id - xx^\top,
\]
so that $a(x)=c(x) = \Id - xx^\top$ on $\Scal^{d-1}$. This is the orthogonal projection onto the tangent space of $\Scal^{d-1}$ at $x$. Setting $A_0=0$, we thus recover the following well-known quadratic SDE for Brownian motion on the sphere:
\[
dX_t = (\Id - X_tX_t^\top)\circ dW_t.
\]
The linear SDE in Corollary~\ref{C:sphere} is different, and was originally obtained by \citet{price1983rolling}; see also \citet{van1985brownian}.
\end{example}

\subsection{Pathwise uniqueness} \label{S:pathwise}

It is natural to ask whether polynomial diffusions on the unit ball can be realized as solutions to SDEs for which pathwise uniqueness holds. As mentioned in Remark~\ref{R:pwunique}, if $a(x)$ is positive definite on the interior of $\Bcal^d$, one obtains pathwise uniqueness up to the first hitting time of $\Scal^{d-1}$ by taking the positive semidefinite square root $\sigma(x)=a(x)^{1/2}$.

If ${\rm BQ}(x,y)$ is a sum of squares, more can be said. Indeed, in this case Theorem~\ref{T:SDE} yields pathwise uniqueness up to the first hitting time of $\Scal^{d-1}$, regardless of whether or not $a(x)$ is nonsingular. Moreover, if the state space is $\Scal^{d-1}$, Corollary~\ref{C:sphere} shows that pathwise uniqueness always holds (still under the sum of squares assumption, of course).

Looking at the SDE~\eqref{eq:SDE:odYX} in Theorem~\ref{T:SDE}, there may be some hope that pathwise uniqueness holds globally, not just up the first hitting time of $\Scal^{d-1}$. We have not been able to prove this in general. Instead, we present a partial result relying on the method developed by \citet{deblassie2004uniqueness} for proving pathwise uniqueness for certain SDEs on the unit ball; see also \citet{swart2002pathwise}. Specifically, we consider the following special case of~\eqref{eq:SDE:odYX}:
\begin{equation}\label{eq:SDE scalar}
dX_t = -\kappa X_t\, dt + \nu \sqrt{1-\|X_t\|^2}\, dW_t + A_0X_t\,dt + A_1X_t\circ d\widehat W^1_t + \cdots + A_mX_t\circ d\widehat W^m_t,
\end{equation}
where $A_0,\ldots,A_m\in\Skew(d)$, $W=(W^1,\ldots,W^d)$ and $\widehat W^1,\ldots,\widehat W^m$ are independent Brownian motions, and, crucially, $\kappa$ and $\nu$ are positive scalar constants. Within the polynomial class, this is a generalization of the equation considered by \citet{swart2002pathwise} and \citet{deblassie2004uniqueness}, who take the tangential components $A_0,\ldots,A_m$ to be zero. Note that even in this case, it is not known whether pathwise uniqueness holds for~\eqref{eq:SDE scalar} for arbitrary values of $\kappa$ and $\nu$.

The reason why \citet{deblassie2004uniqueness} method can be applied in the present situation is contained in the following calculation. Define $Y_t=1-\|X_t\|^2$. Then, an application of It\^o's formula and the change-of-variable formula for the Stratonovich integral yield
\[
dY_t = 2\kappa\|X_t\|^2\,dt - 2\nu\sqrt{Y_t} X_t^\top dW_t - d\,\nu^2Y_t\,dt - 2X_t^\top A_0 X_t\,dt - 2\sum_{p=1}^m X_t^\top A_p X_t\, d\widehat W^p_t.
\]
The skew-symmetry of the $A_p$ implies that the quadratic terms $X_t^\top A_pX_t$ all vanish, so that
\begin{equation}\label{eq:dYW}
dY_t = \left(2\kappa\|X_t\|^2 - d\,\nu^2Y_t\right)dt - 2\nu\sqrt{Y_t} X_t^\top dW_t.
\end{equation}
This no longer involves the matrices $A_p$. Note that this crucially relies on the linearity in $X_t$ of the last $m+1$ terms of~\eqref{eq:SDE scalar}. This property is specific for the sum of squares case~\eqref{eq:SDE:odYX}, and is in general not present in~\eqref{eq:SDE}.

\begin{theorem} \label{T:pathwise uniqueness}
Assume $\kappa/\nu^2>\sqrt{2}-1$. Then pathwise uniqueness holds for~\eqref{eq:SDE scalar}.
\end{theorem}

\begin{proof}
Let $X$ and $\widetilde X$ be two solutions to~\eqref{eq:SDE scalar}, driven by the same Brownian motions and starting from the same starting point $x_0\in\Bcal^d$. Pathwise uniqueness clearly holds up to the first time the boundary is attained, so it suffices to take $x_0$ on the boundary, and prove $X_t=\widetilde X_t$ for all $t\le\tau$, where $\tau=\inf\{t\ge0\colon \|X_t\|\wedge\|\widetilde X_t\|\le1-\varepsilon\}$ for some arbitrary $\varepsilon>0$. Define $Y_t=1-\|X_t\|^2$ and $\widetilde Y_t=1-\|\widetilde X_t\|^2$. In order to improve readability we often omit time subscripts below. We also omit the details of several cumbersome but straightforward calculations.

We apply the method by \citet{deblassie2004uniqueness}. Define the function
\[
F(y,\widetilde y) = (y^p - \widetilde y^p)^2
\]
where $p\in(1/2,1)$ is required to satisfy $\kappa>\nu^2(1-p)$. It then follows from Lemma~2.1 in \citet{deblassie2004uniqueness} together with~\eqref{eq:dYW} that $F(Y,\widetilde Y)$ satisfies It\^o's formula on the time interval $[0,\tau]$, provided $\varepsilon>0$ is small enough depending on $\kappa$, $\nu$, and $p$. Writing $F_1 = \partial_y F(Y,\widetilde Y)$, $F_2 = \partial_{\widetilde y} F(Y,\widetilde Y)$, etc., this means that the following calculations are rigorous. First,
\[
dF(Y,\widetilde Y) = F_1\,dY + F_2\,d\widetilde Y + \frac{1}{2}F_{11}\,d\langle Y,Y\rangle + F_{12}\,d\langle Y,\widetilde Y\rangle + \frac{1}{2}F_{22}\,d\langle \widetilde Y,\widetilde Y\rangle.
\]
Next, in view of~\eqref{eq:dYW}, a calculation yields
\begin{align*}
dF(Y,\widetilde Y) &= \text{(local martingale)} \\
&\qquad + \Big(\kappa(F_1+F_2)+\nu^2( F_{11} Y + 2F_{12} Y^{1/2}\widetilde Y^{1/2} +  F_{22}\widetilde Y) \Big) ( \|X\|^2+\|\widetilde X\|^2)\,dt \\
&\qquad + \Big( \kappa(F_1-F_2)+\nu^2( F_{11} Y - F_{22}\widetilde Y) \Big)( \|X\|^2-\|\widetilde X\|^2)\,dt \\
&\qquad - 2\nu^2 F_{12}Y^{1/2}\widetilde Y^{1/2} \|X-\widetilde X\|^2\,dt - d\,\nu^2(F_1 Y + F_2\widetilde Y)\,dt \\
&= \text{(local martingale)} + {\bf I}_1( \|X\|^2+\|\widetilde X\|^2)\,dt + {\bf I}_2\,dt + {\bf I}_3\,dt + {\bf I}_4\,dt.
\end{align*}
We now insert the explicit expression for $F$ and its derivatives. Define $Z=(Y^{p-1}-\widetilde Y^{p-1})(\widetilde Y^p - Y^p)$ and note that Lemma~A.1 in \citet{deblassie2004uniqueness} yields
\[
(Y^{p-1/2}-\widetilde Y^{p-1/2})^2 \le \frac{(2p-1)^2}{4p(1-p)} Z
\]
whenever $p\in(\frac{1}{2},1)$. Making use of this inequality, one obtains
\[
{\bf I}_1 = -2p\kappa Z + 2p\nu^2(1-p) Z + 2p^2\nu^2(Y^{p-1/2}-\widetilde Y^{p-1/2})^2 \le -2p\nu^2\left(\frac{\kappa}{\nu^2} - 1+p - \frac{(2p-1)^2}{4(1-p)}\right)Z.
\]
The expression in parentheses on the right-hand side is maximized by $p=1-\sqrt{2}/4\in(\frac{1}{2},1)$, which makes its value equal to $\kappa/\nu^2+1-\sqrt{2}$. Note that the hypothesis on $\kappa$ and $\nu$ ensures that $\kappa>\nu^2(1-p)$ holds for this choice of $p$, as was required above. Thus, on $[0,\tau]$,
\[
{\bf I}_1 ( \|X\|^2+\|\widetilde X\|^2) \le  - 4(1-\varepsilon)p\nu^2\left(\frac{\kappa}{\nu^2} + 1 - \sqrt{2} \right)Z.
\]
Next, we have
\[
{\bf I}_2 = -2p\nu^2\left(\frac{\kappa}{\nu^2} - 1+p\right)(Y^{p-1}+\widetilde Y^{p-1})(Y^p-\widetilde Y^p)(Y - \widetilde Y) - 2p^2\nu^2(Y^{2p-1}-\widetilde Y^{2p-1})(Y-\widetilde Y).
\]
Since $\kappa/\nu^2\ge 1-p$, one obtains ${\bf I}_2\le 0$. Furthermore, since
\[
-F_{12}Y^{1/2}\widetilde Y^{1/2}=2p^2Y^{p-1/2}\widetilde Y^{p-1/2}\le 2p^2\varepsilon^{2p-1}
\]
for $t\le\tau$, we get ${\bf I}_3 \le 4p^2\varepsilon^{2p-1}\nu^2\|X-\widetilde X\|^2$ on $[0,\tau]$. Finally, ${\bf I}_4=-2pd\,\nu^2 F\le0$. Putting together these estimates yields, for $t\le\tau$,
\begin{equation} \label{eq:F(Y,Ytilde)}
\begin{aligned}
F(Y_t,\widetilde Y_t) &\le \text{(local martingale)} - 4(1-\varepsilon)p\nu^2\left(\frac{\kappa}{\nu^2} + 1 - \sqrt{2} \right) \int_0^t Z_s\,ds \\
&\qquad + 4p^2\varepsilon^{2p-1}\nu^2 \int_0^t \|X_s-\widetilde X_s\|^2\,ds.
\end{aligned}
\end{equation}
We need the dynamics of $\|X-\widetilde X\|^2$. Similar calculations as those leading to~\eqref{eq:dYW} yield
\[
d\|X-\widetilde X\|^2 = \left( -2\kappa \|X-\widetilde X\|^2 + d\,\nu^2(Y^{1/2}-\widetilde Y^{1/2})^2\right)dt + 2\nu (Y^{1/2}-\widetilde Y^{1/2})(X-\widetilde X)^\top dW.
\]
Moreover, one has the inequality $(Y^{1/2}-\widetilde Y^{1/2})^2\le c\varepsilon^{2-2p}Z$ on $[0,\tau]$, where $c$ is a constant that only depends on $p$; see the proof of \cite[Lemma~3.6]{deblassie2004uniqueness}. In conjunction with~\eqref{eq:F(Y,Ytilde)} one then obtains, for $t\le\tau$,
\begin{align*}
F(Y_t,\widetilde Y_t) + \|X_t-\widetilde X_t\|^2 &\le \text{(local martingale)} \\
&\qquad + \left( d\,\nu^2c\varepsilon^{2-2p} - 4(1-\varepsilon)p\nu^2\left(\frac{\kappa}{\nu^2} + 1 - \sqrt{2} \right)\right)  \int_0^t Z_s\,ds\\
&\qquad  + (4p^2\varepsilon^{2p-1}\nu^2 - 2\kappa) \int_0^t \|X_s-\widetilde X_s\|^2\,ds.
\end{align*}
Since $\kappa/\nu^2+1-\sqrt{2}>0$ by assumption, we may choose $\varepsilon$ sufficiently small to make the finite variation terms nonpositive. Thus, on $[0,\tau]$, $F(Y_t,\widetilde Y_t) + \|X_t-\widetilde X_t\|^2$ is bounded from above by a nonnegative local martingale $M_t$ with $M_0=0$. Thus $M_t=0$ on $[0,\tau]$, implying that the same holds for the left-hand side. This yields $X=\widetilde X$ on $[0,\tau]$, as desired. The proof is complete.
\end{proof}

\begin{remark}
While the restriction on $\kappa$ and $\nu$ excludes some specifications, there is certainly an overlap with the set of parameters for which the boundary is attained. Indeed, by Proposition~\ref{P:boundary invariant} the boundary may be attained if and only if $\kappa/\nu^2<1$.
\end{remark}

\subsection{Existence of smooth densities}

If ${\rm BQ}(x,y)$ is a sum of squares and the state space is the unit sphere, Corollary~\ref{C:sphere}\ref{C:sphere:3} and H\"ormander's theorem lead to simple conditions under which the solution to~\eqref{eq:SDE} possesses a smooth density with respect to a suitable surface area measure. To state the precise result we introduce some notation. Let $A_0,\ldots,A_m$ be elements of $\Skew(d)$ and fix $x_0\in\Scal^{d-1}$. Define
\begin{align}
\mathfrak g_0 &= \{A_1,\ldots,A_m\} \nonumber\\
\mathfrak g_k &= \mathfrak g_{k-1} \cup \{[B,A_p] \colon B\in \mathfrak g_{k-1}, \ p=0,\ldots,m\} \qquad (k\ge 1) \nonumber\\
\mathfrak g &= {\rm span}\bigcup_{k\ge 0} \mathfrak g_k  \label{eq:lie g} \\
\mathfrak h &= \text{Lie algebra generated by $A_0,\ldots,A_m$}. \nonumber
\end{align}
Here $[A,B]=AB-BA$ is the usual matrix commutator. Note that $\mathfrak g$ is a Lie sub-algebra of $\mathfrak h$, and that $\mathfrak h=\mathfrak g+\R A_0$. In fact, $\mathfrak g$ is even an ideal in $\mathfrak h$, meaning that $[\mathfrak g,\mathfrak h]\subseteq\mathfrak g$. Let $\bm G$ denote the connected Lie subgroup of $SO(d)$ generated by $e^B$, $B\in\mathfrak g$, and $\bm H$ the connected Lie group similarly generated by $\mathfrak h$. Then $\mathfrak g$ (respectively $\mathfrak h$) is the Lie algebra of $\bm G$ (respectively $\bm H$). Finally, define $\bm Gx_0=\{Qx_0\colon Q\in\bm G\}$, the orbit of $x_0$ under $\bm G$, and similarly $\bm Hx_0$. These are smooth submanifolds of $\Scal^{d-1}$, and their tangent spaces are given by
\[
T_x(\bm Gx_0) = \mathfrak g x \qquad\text{and}\qquad T_x(\bm Hx_0) = \mathfrak h x.
\]
In view of the identities $Q^{-1}\mathfrak g Q=\mathfrak g$ and $Q^{-1}\mathfrak h Q=\mathfrak h$ for all $Q\in\bm G$, it follows that
\begin{equation} \label{eq:TG TH}
T_{Qx}(\bm Gx_0) = Q\,T_x(\bm G x_0) \quad\text{and}\quad T_{Qx}(\bm Hx_0) = Q\,T_x(\bm H x_0), \qquad Q \in \bm G.
\end{equation}
All these properties are well-known, and simply reflect the fact that $\bm Gx_0$ and $\bm Hx_0$ are homogeneous spaces for $\bm G$ and $\bm H$, respectively.

The natural state space for solutions to \eqref{eq:odYX} starting from $x_0\in\Scal^{d-1}$ is ${\bm H} x_0$. However, by slightly adjusting such a solution one obtains a process which remains in $\bm Gx_0$. Specifically, one has the following lemma.

\begin{lemma} \label{L:Zt inv}
Let $X$ be the solution to~\eqref{eq:odYX} with $X_0=x_0$, and define a process $Z$ by
\[
Z_t = e^{-tA_0} X_t.
\]
Then $Z_t \in \bm Gx_0$ for all $t\ge0$.
\end{lemma}

\begin{proof}
Since $\mathfrak g$ is an ideal in $\mathfrak h$, we have for each $p\in\{1,\ldots,m\}$,
\begin{equation} \label{eq:z tangent}
e^{-tA_0}A_p e^{tA_0}z \in \mathfrak g z = T_z(\bm Gx_0) \quad\text{for all}\quad z\in \bm Gx_0, \quad t\ge0.
\end{equation}
To see this, write $f(t)=e^{-tA_0}A_p e^{tA_0}$. Then $f^{(k+1)}(0)=[f^{(k)}(0),A_0]$ for all $k\in\N$. Since $\mathfrak g$ is an ideal in $\mathfrak h$ one has $f'(0)=[A_p,A_0]\in\mathfrak g$, whence $f^{(k)}(0)\in\mathfrak g$ for all $k\in\N$ by induction. Thus $e^{-tA_0}A_pe^{tA_0}=f(t)\in\mathfrak g$ since $f$ is analytic with infinite radius of convergence. Hence \eqref{eq:z tangent} follows. Now, a brief calculation shows that $Z$ satisfies the time-inhomogeneous SDE
\[
dZ_t = e^{-tA_0}A_1e^{tA_0} Z_t \circ dW_t^1 + \cdots + e^{-tA_0}A_me^{tA_0} Z_t \circ dW_t^m,
\]
which admits a unique strong solution. In view of~\eqref{eq:z tangent}, the vector fields on the right-hand side are tangent to $\bm Gx_0$, which by \citet[Theorem~1.2.9]{hsu2002stochastic} implies that $Z$ takes values there.
\end{proof}

One now has the following fairly precise result regarding existence of smooth densities in the case where ${\rm BQ}(x,y)$ admits a sum of squares representation.

\begin{theorem} \label{T:sphere density}
Let $X$ be the solution to~\eqref{eq:odYX} with $X_0=x_0\in\Scal^{d-1}$; in particular, ${\rm BQ}(x,y)$ is a sum of squares of polynomials. Then $X$ takes values in $\bm H x_0$. Moreover, the following conditions are equivalent:
\begin{enumerate}
\item\label{T:sd:density} $X_t$ has a smooth density with respect to surface area measure on $\bm H x_0$ for all $t>0$.
\item\label{T:sd:H=G} $\bm H x_0=\bm G x_0$.
\item\label{T:sd:Axg} $A_0x_0 \in \mathfrak g x_0$.
\end{enumerate}
\end{theorem}

\begin{proof}
It follows from Lemma~\ref{L:Zt inv} and the definition of $\bm H$ that $X$ takes values in $\bm H x_0$. We now prove that the stated conditions are equivalent. To this end, first note that since $\bm Gx_0\subseteq\bm Hx_0$, and $\bm Hx_0$ is connected due to the connectedness of $\bm H$, one has $\bm G x_0 = \bm H x_0$ if and only if $T_x(\bm Gx_0)=T_x(\bm Hx_0)$ for all $x\in \bm Gx_0$. Thanks to \eqref{eq:TG TH} this yields
\begin{align}
\bm G x_0 = \bm H x_0 &\qquad\Longleftrightarrow\qquad \text{$T_x(\bm Gx_0)=T_x(\bm Hx_0)$ for some $x\in \bm Gx_0$}.  \label{eq:GxHxequiv0} \\
&\qquad\Longleftrightarrow\qquad \text{$T_{x_0}(\bm Gx_0)=T_{x_0}(\bm Hx_0)$}. \nonumber
\end{align}
The latter condition is equivalent to $\mathfrak g x_0 = \mathfrak h x_0$. Since $\mathfrak h=\mathfrak g+\R A_0$, this holds if and only if $A_0x_0\in \mathfrak gx_0$. This proves \ref{T:sd:H=G} $\Longleftrightarrow$ \ref{T:sd:Axg}.

To prove \ref{T:sd:density} $\Longrightarrow$ \ref{T:sd:H=G}, let $\sigma$ denote surface area measure on $\bm H x_0$ and note that for any measurable subset $U\subseteq \bm Hx_0$, one has $\sigma(U)>0$ if and only if $\sigma(e^{tA_0}U)>0$, where $t\in\R$ is arbitrary. Fix some $t>0$. Since Lemma~\ref{L:Zt inv} yields $X_t\in e^{t A_0}\bm G x_0$, it follows that $\sigma(e^{t A_0}\bm G x_0)>0$ and hence $\sigma(\bm G x_0)>0$. This in turn implies that there exists some $x\in \bm Gx_0$ for which the tangent space $T_x(\bm Gx_0)$ cannot be a proper subspace of $T_x(\bm Hx_0)$; that is, $T_x(\bm Gx_0)=T_x(\bm Hx_0)$. Thus \ref{T:sd:H=G} follows by~\eqref{eq:GxHxequiv0}

It remains to prove \ref{T:sd:H=G} $\Longrightarrow$ \ref{T:sd:density}. Consider the vector fields $V_p$ given by $V_p(x)=A_px$, appearing in the expression~\eqref{eq:Horm} for the generator $\Gcal$. One has $[V_p,V_q](x)=[A_p,A_q]x$, where $[A_p,A_q]=A_pA_q-A_qA_p$. H\"ormander's parabolic condition (see \citet{hormander1967hypoelliptic}, or Theorem~1.3 in~\citet{hairer2011malliavin} for a formulation suited to our current setting) is therefore equivalent to requiring $T_x(\bm H x_0) = \mathfrak g x$ for all $x\in \bm H x_0$. This holds thanks to the assumption $\bm H x_0=\bm Gx_0$. Applying H\"ormander's theorem in each smooth coordinate chart of $\bm Hx_0$ now gives the existence of a smooth density.
\end{proof}

\begin{corollary} \label{C:sphere density}
Let $X$ be the solution to~\eqref{eq:odYX} with $X_0=x_0\in\Scal^{d-1}$. Then $X_t$ has a smooth density with respect to surface area measure on $\Scal^{d-1}$ for all $t>0$ if and only if $\mathfrak g = \Skew(d)$.
\end{corollary}

\begin{example}
As a simple illustration of the distinction between $\bm G$ and $\bm H$, let $d=4$ and consider the skew-symmetric matrices
\[
A_0 = \begin{pmatrix}0&1&0&0\\-1&0&0&0\\0&0&0&0\\0&0&0&0\end{pmatrix},
\qquad\qquad
A_1 = \begin{pmatrix}0&0&0&0\\0&0&0&0\\0&0&0&1\\0&0&-1&0\end{pmatrix}.
\]
Then we have $A_0A_1=A_1A_0=0$, so that $\mathfrak g={\rm span}\{A_1\}$ and $\mathfrak h={\rm span}\{A_0,A_1\}$. Moreover, $e^{-tA_0}A_1e^{tA_0}=A_1$, whence the process $Z$ in Lemma~\ref{L:Zt inv} satisfies $dZ_t=A_1Z_t\circ dW_t^1$. It follows that $(Z^3_t)^2+(Z^4_t)^2$ is constant, showing that $Z$ is confined to a circle embedded in $\R^4$, namely ${\bm G}x_0=e^{\R A_1}x_0$. On the other hand, the process $X$ given by $X_t=e^{tA_0}Z_t$ is not restricted to any fixed subset of ${\bm H}x_0$, although for a fixed time $t\ge 0$ it necessarily lies in the transformed circle $e^{tA_0}\bm G x_0$ (we suppose that $A_0x_0\ne 0$ to avoid a trivial situation). This is a nullset in $\bm H x_0$, so a density cannot exist.
\end{example}

By means of a projection argument, Corollary~\ref{C:sphere density} can be used to derive conditions that guarantee the existence of smooth densities for a particular class of polynomial diffusions taking values in the ball $\Bcal^d$.

\begin{corollary}
Let $X$ be the solution to~\eqref{eq:SDE:odYX} with $X_0=x_0\in\Scal^{d-1}$ and $\widehat B=-\frac 12\alpha$. Write $\alpha=a_1a_1^\top + \cdots +a_da_d^\top$ for some vectors $a_1,\ldots,a_d\in\R^d$, and define matrices $\widetilde A_p\in\Skew(d+1)$ by
\[
\widetilde A_p = \begin{pmatrix}A_p & 0 \\ 0 & 0 \end{pmatrix} \text{ for } p=0,\ldots,m,
\quad\text{and}\quad
\widetilde A_p = \begin{pmatrix}0 & a_i \\ -a_i^\top & 0 \end{pmatrix} \text{ for } p=m+i,\ i=1,\ldots,d.
\]
Let $\widetilde{\mathfrak g}$ be defined as in~\eqref{eq:lie g} with $A_p$ replaced by $\widetilde A_p$. If $\widetilde{\mathfrak g}=\Skew(d+1)$, then for all $t>0$, $X_t$ has a density that is smooth on the interior of $\Bcal^d$.
\end{corollary}

\begin{proof}
Let $z_0=(x_0,\sqrt{1-\|x_0\|^2})\in\Scal^d$, and define the $\Scal^{d}$-valued process $Z$ as the unique solution to the SDE
\[
dZ_t = \left(\widetilde A_0dt + (\circ\,d\widetilde W^1_t)\, \widetilde A_1 + \cdots (\circ\,d\widetilde W^{m+d}_t)\, \widetilde A_{m+d}\right) Z_t, \qquad Z_0=z_0,
\]
where  $\widetilde W=(\widetilde W^1,\ldots,\widetilde W^{m+d})$ is an $(m+d)$-dimensional Brownian motion. Let
\[
\pi:\Scal^d\to\R^d, \quad (z_1,\ldots,z_d,z_{d+1})\mapsto (z_1,\ldots,z_d)
\]
be the canonical projection of $\Scal^d$ onto $\R^d$, and define $X'=\pi(Z)$. Then $X'$ satisfies
\[
d X_t' = A_0 X_t'\,dt + Z_{d+1,t} \sum_{i=1}^d a_i\circ d\widetilde W^{m+i}_t + \Big( \sum_{p=1}^m A_p\circ d\widetilde W^p_t\Big) X_t'.
\]
Moreover, since
\[
dZ_{d+1,t} = -\sum_{i=1}^d a_i^\top X'_t \circ d\widetilde W^{m+i}_t,
\]
one has $\langle Z_{d+1},\widetilde W^{m+i}\rangle_t = - a_i^\top X'_t dt$. Thus, writing the middle term in the above expression for $dX'_t$ in It\^o form, one obtains
\[
d X_t' = \left(A_0 - \frac{1}{2}\alpha\right) X_t'\,dt + Z_{d+1,t} \sum_{i=1}^d a_i\, d\widetilde W^{m+i}_t + \Big( \sum_{p=1}^m A_p\circ d\widetilde W^p_t\Big) X_t'.
\]
Note that $Z_{d+1,t} = \sqrt{1-\|X_t'\|^2}\,{\rm sign}(Z_{d+1,t})$. Thus $X'$ satisfies an SDE of the form~\eqref{eq:SDE}, and its generator coincides with that of $X$. Thus $X$ and $X'$ have the same law by Lemma~\ref{L:law}.

Next, let $\sigma$ be surface area measure on $\Scal^d$. By Corollary~\ref{C:sphere density}, $Z_t$ has a smooth density $p_Z(t;z)$ with respect to~$\sigma$ for all $t>0$. We claim that this implies that $X_t'$ has a density, smooth on the interior of $\Bcal^d$, for every $t>0$. Since $X$ and $X'$ have the same law, this will complete the proof of the corollary.

Let $\Scal^d_+=\{z\in\Scal^d\colon z_{d+1}>0\}$ denote the upper hemisphere of $\Scal^d$. Let $\mu=\sigma|_{\Scal^d_+}\circ\pi^{-1}$ be the pushforward under $\pi$ of the restriction of $\sigma$ to $\Scal^d_+$, and recall that one has
\[
\mu(dx) = (1-\|x\|^2)^{-1/2}\,dx, \qquad x\in\Bcal^d.
\]
For any measurable subset $A\subseteq\Bcal^d$ we compute:
\begin{align*}
\P(X_t'\in A) &= \P(Z_t\in\pi^{-1}(A)) = \int_{\pi^{-1}(A)} p_Z(t;z)\sigma(dz) \\
&= \int_{\Scal^d_+\cap \pi^{-1}(A)} \big( p_Z(t;z_1,\ldots,z_d,z_{d+1}) + p_Z(t;z_1,\ldots,z_d,-z_{d+1}) \big) \sigma(dz) \\
&= \int_A \left( p_Z(t;x,\sqrt{1-\|x\|^2}) + p_Z(t;x,-\sqrt{1-\|x\|^2}) \right) \mu(dx).
\end{align*}
Thus $X'_t$, hence $X_t$, has a density $p_X(t;x)$ given by
\[
p_X(t;x) = \left( p_Z(t;x,\sqrt{1-\|x\|^2}) + p_Z(t;x,-\sqrt{1-\|x\|^2}) \right) (1-\|x\|^2)^{-1/2},
\]
which is smooth on the interior of $\Bcal^d$.
\end{proof}

\section{Proofs for Section~\ref{S:sos}} \label{S:algebra}

In this section we collect the proofs of the main results in Section~\ref{S:sos}. In particular, the proof of Theorem~\ref{T:sos} relies on results regarding the space $\Kcal$ defined in~\eqref{eq:DK}; we develop these results in Section~\ref{S:setK}.

\subsection{Proof of Theorem~\ref{T:charC}} \label{S:proof_charC}

Recall that we set ${d \choose k}=0$ for $k>d$. Consider the linear spaces
\begin{align*}
\Dcal &= \left\{ c: \R^d \to \S^d : c_{ij} \in \Hom_2 \text{ for all } i,j\right\} \\
\Ecal &= \left\{ f: \R^d \to \R^d : f_i \in \Hom_3 \text{ for all } i\right\}
\end{align*}
as well as the linear map $T:\Dcal\to\Ecal$ defined by $(Tc)(x)=c(x)x$. Then $\ker T=\Ccal$, so we need to prove that $\ker T=\{c_H:H\in\S^d\}$. We make the following observations.
\begin{enumerate}
\item $\dim \Dcal = {d+1\choose 2}^2$, since each $c\in\Dcal$ is specified by ${d+1\choose 2}$ independently chosen elements of $\Hom_2$, each of which is specified by $\dim\Hom_2={d+1\choose 2}$ independent parameters.
\item $\dim\Ecal = d{d+2\choose 3}$, since each $f\in\Ecal$ is specified by $d$ independently chosen elements of $\Hom_3$, each of which is specified by $\dim\Hom_3={d+2\choose 3}$ independent parameters.
\item\label{T:charC:2} $\dim \Dcal - \dim \Ecal = 2{d\choose 4} + 3{d\choose 3} + {d\choose 2}$. This follows from a direct calculation.
\item\label{T:charC:3} $T$ is surjective. To see this, fix any $i,j,k,l\in\{1,\ldots,d\}$. We need to find $c\in\Dcal$ such that $c(x)x=x_ix_jx_ke_l$. One such $c$ is given by $c(x)=x_i( x_jE_{kl} + x_kE_{jl} - x_lE_{jk})/2$, where we define $E_{st}=e_se_t^\top + e_te_s^\top$.
\end{enumerate}
Since the collection $\{D_p\,xx^\top D_q^\top + D_q\,xx^\top D_p^\top : 1\le p<q\le m\}$ is a subset of $\ker T$, the rank-nullity theorem together with \ref{T:charC:2} and \ref{T:charC:3} implies~\eqref{T:charC:1} once we prove that this collection contains $2{d\choose 4} + 3{d\choose 3} + {d\choose 2}$ linearly independent elements. To this end, recall the elementary skew-symmetric matrices
\[
S_{ij} = e_ie_j^\top - e_je_i^\top
\]
appearing in~\eqref{eq:SijD}, and consider the sub-collection
\begin{align*}
\Bcal \ =\ &\{ S_{ij}\,xx^\top S_{kl}^\top + S_{kl}\,xx^\top S_{ij}^\top : 1\le i<j<k<l\le d \} \\
\cup &\{ S_{ik}\,xx^\top S_{jl}^\top + S_{jl}\,xx^\top S_{ik}^\top : 1\le i<j<k<l\le d \} \\
\cup &\{ S_{ij}\,xx^\top S_{ik}^\top + S_{ik}\,xx^\top S_{ij}^\top : 1\le i<j<k\le d \} \\
\cup &\{ S_{ij}\,xx^\top S_{jk}^\top + S_{jk}\,xx^\top S_{ij}^\top : 1\le i<j<k\le d \} \\
\cup &\{ S_{ik}\,xx^\top S_{jk}^\top + S_{jk}\,xx^\top S_{ik}^\top : 1\le i<j<k\le d \} \\
\cup &\{ 2 S_{ij}\,xx^\top S_{ij}^\top : 1\le i<j\le d \}.
\end{align*}
It contains $2{d\choose 4} + 3{d\choose 3} + {d\choose 2}$ elements, which as we now show are linearly independent. For any $i<j$, the matrix $2S_{ij}\,xx^\top S_{ij}^\top$ is the only matrix in $\Bcal$ with $x_i^2$ appearing in position $(j,j)$. Moreover, for any $i<j<k$, $S_{ij}\,xx^\top S_{jk}^\top + S_{jk}\,xx^\top S_{ij}^\top$ is the only matrix in $\Bcal$ with $x_j^2$ in position $(i,k)$; $S_{ik}\,xx^\top S_{jk}^\top + S_{jk}\,xx^\top S_{ik}^\top$ is the only matrix with $x_k^2$ in position $(i,j)$; and $S_{ij}\,xx^\top S_{ik}^\top + S_{ik}\,xx^\top S_{ij}^\top$ is the only matrix with $x_i^2$ in position $(j,k)$. Finally, for any $i<j<k<l$, $S_{ij}\,xx^\top S_{kl}^\top + S_{kl}\,xx^\top S_{ij}^\top$ is the only matrix in $\Bcal$ with $x_ix_l$ in position $(j,k)$; and $S_{ik}\,xx^\top S_{jl}^\top + S_{jl}\,xx^\top S_{ik}^\top$ is the only matrix with $x_kx_l$ in position $(i,j)$. These observations imply that $\Bcal$ is a linearly independent set, as claimed. Hence~\eqref{T:charC:1} is proved.

It remains to show that $\dim\Kcal={d\choose 4}$. To see this, observe that $\Kcal$ is the kernel of the linear map
\[
\S^m \to \Ccal, \qquad H\mapsto c_H.
\]
By \eqref{T:charC:1}, this map is surjective, and the dimension of its kernel is
\[
\dim\S^m-\dim\Ccal={m+1\choose 2}-2{d\choose 4} - 3{d\choose 3} - {d\choose 2}={d\choose 4},
\]
where the last equality follows from a direct calculation.

\subsection{The space $\Kcal$ and the Pl\"ucker relations}\label{S:setK}

The goal is now to develop a better understanding of the space $\Kcal$ in~\eqref{eq:DK}. The main outcome is Lemma~\ref{L:M=N} below, which is crucial for the proof of Theorem~\ref{T:sos}. The lemma shows in particular that the set $\Kcal$ can be identified with the Pl\"ucker relations from algebraic geometry.

Let $m={d\choose 2}=\dim\Skew(d)$. Each matrix $H=(h_{pq})\in\S^m$ can be viewed as a symmetric linear map on $\Skew(d)$, acting on matrices $A=(a_{ij})\in\Skew(d)$ by the formula
\[
H[A]_{ij} = \sum_{k<l} h_{\pi(i,j),\pi(k,l)}a_{kl}, \qquad i<j,
\]
with the remaining components of $H[A]$ given by skew-symmetry. Here the map $\pi$ lists strictly upper-triangular matrix elements in lexicographic order,
\[
\begin{array}{r|cccccccccc}
(i,j)	&(1,2)	&(1,3)	&(1,4)	&\cdots&(1,d)	&(2,3)	&\cdots&(2,d)	&\cdots&(d-1,d) \\[1ex]
\hline \\[-2ex]
\pi(i,j)&1 		&2		&3		&\cdots&d-1	&d		&\cdots&2d-3	&\cdots&m={d\choose 2}
\end{array}
\]
The symmetry property $\langle A,H[B]\rangle=\langle H[A],B\rangle$ is immediate from the symmetry of $H$ as a matrix. Using this notation, the biquadratic form corresponding to an element $c_H\in\Ccal$, where $H\in\S^m$, can be expressed as
\begin{equation} \label{eq:yCHxy}
y^\top c_H(x) y = \frac{1}{2}\langle A, H[A] \rangle \qquad \text{where}\qquad A=xy^\top - yx^\top.
\end{equation}
This follows from~\eqref{eq:cH(x)} using that $x^\top D_{\pi(i,j)}y=x_iy_j-x_jy_i$. Together with Lemma~\ref{L:rank2} below this immediately yields
\begin{equation} \label{eq:KSkew2}
\Kcal = \left\{ K\in \S^m \colon \langle A, K[A]\rangle = 0 \text{ for all } A\in\Skew(2,d) \right\}.
\end{equation}
In particular, this shows that $\Ccal \cong \S^m/\Kcal$ can be identified with the set of restrictions to $\Skew(2,d)$ of quadratic forms on $\Skew(d)$.

\begin{lemma} \label{L:rank2}
Let $A\in\Skew(d)$. Then $\rk A=2$ if and only if $A=xy^\top-yx^\top$ for some linearly independent $x,y\in\R^d$.
\end{lemma}

\begin{proof}
Only the necessity needs proof, so pick $A\in\Skew(2,d)$. Any rank-two matrix $A$ can be written $A=xx^\top A + yy^\top A$, where $\{x,y\}$ is an orthonormal basis for the range of $A$. By skew-symmetry of $A$, we have $A^\top x=\gamma y$ and $A^\top y=-\gamma x$, where $\gamma=x^\top Ay$. Thus $A=(\gamma x)y^\top - y(\gamma x)^\top$, which is of the stated form.
\end{proof}

For each ordered $4$-tuple $i<j<k<l$ of indices from $\{1,\ldots,d\}$, define a polynomial $P_{ijkl}$ in the entries of $A\in\Skew(d)$ by
\[
P_{ijkl}(A) = a_{ij}a_{kl} - a_{ik}a_{jl} + a_{il}a_{jk}.
\]
These polynomials generate the variety of rank-two skew-symmetric matrices: all $P_{ijkl}$ vanish at a nonzero $A\in\Skew(d)$ if and only if $\rk A=2$. Moreover, they constitute a basis for the homogeneous ideal generated by this variety: any homogeneous polynomial vanishing on~$\Skew(2,d)$ is a polynomial linear combination of the $P_{ijkl}$. The classical use of this fact is to give the Grassmannian manifold of two-dimensional subspaces of $\R^d$ the structure of a projective variety. This is achieved via the {\em Pl\"ucker embedding}, which maps a subspace $V={\rm span}\{x,y\}$ of~$\R^d$ to $A=xy^\top-yx^\top \in\Skew(d)$. Note that, up to scaling, $A$ only depends on $V$, not on the choice of spanning vectors $x$ and $y$. Conversely, any $A\in\Skew(2,d)$ corresponds to a unique two-dimensional subspace $V\subset\R^d$ due to Lemma~\ref{L:rank2}. The entries $a_{ij}$ of $A$ are called the {\em Pl\"ucker (or Grassmann) coordinates} of~$V$, and the quadratic equations $P_{ijkl}(A)=0$ that they satisfy are called the {\em Pl\"ucker relations}. For us, the relevant result is the following:

\begin{lemma} \label{L:Pijkl}
The polynomials $P_{ijkl}$ have the following properties:
\begin{enumerate}
\item\label{L:Pijkl:1} $\Skew(2,d) = \left\{ A \in \Skew(d)\setminus\{0\} \colon P_{ijkl}(A)=0 \text{ for all } i<j<k<l\right\}$.
\item\label{L:Pijkl:2} Let $P(A)$ be a homogeneous polynomial in the entries of $A\in\Skew(d)$. Then $P(A)=0$ for all $A\in\Skew(2,d)$ if and only if it is of the form
\[
P(A) = \sum_{i<j<k<l} Q_{ijkl}(A)P_{ijkl}(A)
\]
for some polynomials $Q_{ijkl}$.
\end{enumerate}
\end{lemma}

\begin{proof}
Keeping in mind the identification of $\Skew(2,d)$ with the set of two-dimensional subspaces as above, part~\ref{L:Pijkl:1} follows from Theorem~II on page~312 in~\citet{hodge1994methods}, and part~\ref{L:Pijkl:2} follows from Theorem~I on page~315 in the same reference.
\end{proof}

The following lemma is the primary reason for studying $\Kcal$ in such detail; it is a crucial ingredient in the proof of Theorem~\ref{T:sos}. In particular, it shows that the set $\Kcal$ can be identified with the Pl\"ucker relations.

\begin{lemma} \label{L:M=N}
The set $\Kcal$ can be identified with the linear span of the polynomials $P_{ijkl}$, $i<j<k<l$. Moreover, one has
\begin{equation} \label{eq:L:M=N}
\Skew(2,d)=\left\{A\in\Skew(d)\setminus\{0\} \colon \langle A,K[A]\rangle=0 \text{ for all } K\in\Kcal\right\}.
\end{equation}
\end{lemma}

\begin{proof}
Lemma~\ref{L:Pijkl}\ref{L:Pijkl:2} together with~\eqref{eq:KSkew2} implies that $K\in\Kcal$ holds if and only if the polynomial $P(A)=\langle A,K[A]\rangle$ is a linear combination of the $P_{ijkl}(A)$. The latter polynomials can therefore be identified with a basis for $\Kcal$. It follows that the right-hand side of~\eqref{eq:L:M=N} is equal to $\left\{ A \in \Skew(d)\setminus\{0\} \colon P_{ijkl}(A)=0 \text{ for all } i<j<k<l\right\}$. This, in turn, is equal to $\Skew(2,d)$ by Lemma~\ref{L:Pijkl}\ref{L:Pijkl:1}.
\end{proof}

\begin{remark} \label{R:Kbasis}
The identification of $\Kcal$ with the linear span of the polynomials $P_{ijkl}$ leads to an explicit description of a basis for $\Kcal$. Indeed, one has
\[
P_{ijkl}(A) = \frac{1}{4}\langle A, K_{(i,j,k,l)}[A]\rangle,
\]
where $K_{(i,j,k,l)}$ is the linear map on $\Skew(d)$ characterized by
\begin{align*}
K_{(i,j,k,l)}[A]_{ij}	&=a_{kl}\\
K_{(i,j,k,l)}[A]_{ik}	&=-a_{jl}\\
K_{(i,j,k,l)}[A]_{il}	&=a_{jk}\\
K_{(i,j,k,l)}[A]_{jk}	&=a_{il}\\
K_{(i,j,k,l)}[A]_{jl}	&=-a_{ik}\\
K_{(i,j,k,l)}[A]_{kl}	&=a_{ij}\\
K_{(i,j,k,l)}[A]_{rs}	&=0\text{ if $\{r,s\}\not\subseteq\{i,j,k,l\}$.}
\end{align*}
The collection of all such $K_{(i,j,k,l)}$ for $i<j<k<l$ thus constitutes a basis for $\Kcal$. Note that there are ${d\choose 4}$ ways of choosing the indices $i,j,k,l$, which is in agreement with the dimension of $\Kcal$ as given in Theorem~\ref{T:charC}.
\end{remark}

\subsection{Proof of Theorem~\ref{T:sos}\ref{T:sos:1}}  \label{S:proof_SOS1}

A preliminary result is required for the proof of Theorem~\ref{T:sos}\ref{T:sos:1}. Recall that an {\em affine subspace} of a vector space is a subset $\Acal$ such that $x,y\in\Acal$, $\lambda\in\R$ implies $\lambda x+(1-\lambda)y\in\Acal$. If $\Acal$ is an affine subspace and $x\in\Acal$, then $\Acal-x$ is a linear subspace that does not depend on~$x$. The {\em dimension} of $\Acal$ is the dimension of $\Acal-x$. The following result is Theorem~1.1 in \citet{barvinok2001remark}; see Theorem~1.3 and its reformulation~(2.2) in \citet{barvinok1995problems} for a proof.

\begin{lemma}\label{L:rank1}
Let $\Acal\subset\S^m$ be an affine subspace such that the intersection $\Acal\cap\S^m_+$ is nonempty and $\codim\Acal\le{r+2\choose 2}-1$ for some $r\in\N$. Then there exists some $B\in\Acal\cap\S^m_+$ with $\rk B\le r$.
\end{lemma}

We now proceed with the proof of Theorem~\ref{T:sos}\ref{T:sos:1}. We need to show that, for any $H\in\S^m$,
\begin{equation}\label{eq:dleq4:1}
\text{$y^\top c_H(x) y\ge 0$ for all $x,y\in\R^d$} \qquad\text{implies}\qquad\text{$(H+\Kcal)\cap\S^m_+\ne\emptyset$.}
\end{equation}
Indeed, then any $c_H\in\Ccal_+$ is equal to $c_{H+K}$ for some $K\in\Kcal$ such that $H+K\in\S^m_+$.

We prove the contrapositive. Thus, consider any $H\in\S^m$ such that $(H+\Kcal)\cap\S^m_+ =\emptyset$. We need to prove that $y^\top c_H(x) y<0$ for some $(x,y)\in\R^d\times\R^d$. The separating hyperplane theorem gives some $B\in\S^m$ such that $\langle C,B\rangle\ge 0$ for all $C\in\S^m_+$ and $\langle H+K,B\rangle<0$ for all $K\in\Kcal$. The former inequality yields $B\in\S^m_+$ by self-duality of $\S^m_+$. The latter inequality yields $\langle K,B\rangle =0$ for all $K\in\Kcal$, since otherwise $\langle H+tK,B\rangle$ would become positive for some $K\in\Kcal$ and some sufficiently large~$t$. Thus, after scaling $B$ if necessary, we have $B\in\Kcal^\perp\cap\S^m_+$ and $\langle H,B\rangle=-1$. Hence $\Acal\cap\S^m_+\ne\emptyset$ for the affine subspace
\[
\Acal = \left\{ C\in \S^m \colon C\in\Kcal^\perp \text{ and } \langle C,H\rangle = -1 \right\}.
\]
Note that $\Acal-B=\{C\in\S^m \colon C\in\Kcal^\perp \text{ and } \langle C,H\rangle = 0\}$. Our initial assumption excludes $H\in\Kcal$, since otherwise we would have $0\in(H+\Kcal)\cap\S^m_+$. Thus $\codim\Acal=1+\dim\Kcal$, so that $\codim\Acal=2$ for $d=4$, and $\codim\Acal=1$ for $d\in\{2,3\}$. In either case, $\codim\Acal\le{3\choose 2}-1$. Applying Lemma~\ref{L:rank1} with $r=1$ then yields a rank-one matrix $zz^\top$, $z\in\R^m\setminus\{0\}$, in the intersection $\Acal\cap\S^m_+$. In particular, $z^\top K z=0$ for all $K\in\Kcal$. Identifying $z$ with the skew-symmetric matrix $A$ whose entries are given by $a_{ij}=z_{\pi(i,j)}$ for $i<j$, this says that $\langle A,K[A]\rangle=0$ for all $K\in\Kcal$. Thus $\rk A=2$ by Lemma~\ref{L:M=N}, whence $A=xy^\top-yx^\top$ for some $x,y\in\R^d$ by Lemma~\ref{L:rank2}. Consequently, \eqref{eq:yCHxy} yields $y^\top c_H(x)y=\langle A,H[A]\rangle/2=z^\top H z=\langle zz^\top,H\rangle=-1$, as required. Theorem~\ref{T:sos}\ref{T:sos:1} is proved.

\begin{remark}
A key step of the proof is to apply Lemma~\ref{L:rank1} to obtain a rank one element of $\Acal\cap\S^m_+$. This is where the proof breaks down for $d=5$. Indeed, Lemma~\ref{L:rank1} then only provides a rank three element, which is not enough to continue the proof.
\end{remark}

\subsection{Proof of Theorem~\ref{T:sos}\ref{T:sos:2}} \label{S:T:sos:2}

We will use the construction given in~\citet{laszlo2010sum} to obtain an example for $d=6$. In order to make the presentation self-contained, we summarize the construction here. Note that it suffices to consider the case $d=6$. Indeed, suppose $f(x_1,\ldots,x_6,y_1,\ldots,y_6)$ is a nonnegative biquadratic form in $6+6$ variables vanishing on the diagonal that is not a sum of squares. For $d>6$, define 
\[
g(x,y)=f(x_1,\ldots,x_6,y_1,\ldots,y_6).
\]
The left-hand side is then a nonnegative biquadratic form in $d+d$ variables vanishing on the diagonal that is not a sum of squares.

From now on we thus take $d=6$ and hence $m=15$. In view of Lemma~\ref{L:sos} we must find $H\in\S^{15}$ such that $y^\top c_H(x) y\ge0$ for all $x,y$, but at the same time
\begin{equation}\label{eq:HpKcSe}
(H+\Kcal)\cap\S^{15}_+=\emptyset;
\end{equation}
in other words, we must find a counterexample to~\eqref{eq:dleq4:1}.  A candidate $H$ is obtained from the inequality
\begin{equation} \label{eq:BQineq}
2(\|X\|^2\|Y \|^2-\tr (X^\top Y )^2)-\|XY-YX\|^2 \ge 0,
\end{equation}
which is valid for any $d\times d$ matrices~$X$ and~$Y$; see for instance \citet[Theorem~2.2]{bottcher2008frobenius}. Here $\|X\|=\sqrt{\tr(X^\top X)}$ is the Frobenius norm. Note also that the left-hand side defines a biquadratic form vanishing on the diagonal. Consider now the special case where for $x,y\in\R^6$,
\[
X=\begin{pmatrix}x_1&x_2&x_3\\0&x_6&x_4\\0&0&x_5\end{pmatrix},\qquad
Y=\begin{pmatrix}y_1&y_2&y_3\\0&y_6&y_4\\0&0&y_5\end{pmatrix}.
\]
As in Section~\ref{S:setK} we use the notation
\[
A=xy^\top-yx^\top.
\]
One then easily verifies the identities $2(\|X\|^2\|Y \|^2-\tr^2 (X^\top Y ))=\|A\|^2$ and $\|XY-YX\|^2=(a_{12}+a_{26})^2+(a_{45}-a_{46})^2+(a_{13}+a_{24}+a_{35})^2$. Furthermore, Lemma~\ref{L:Pijkl}\ref{L:Pijkl:1} yields $P_{1234}(A)=P_{2345}(A)=0$. In view of~\eqref{eq:BQineq} it follows that the function
\[
f(x,y)=\|A\|^2-(a_{12}+a_{26})^2-(a_{45}-a_{46})^2-(a_{13}+a_{24}+a_{35})^2-P_{1234}(A)-P_{2345}(A)
\]
is a nonnegative biquadratic form vanishing on the diagonal. The correspondence established in~\eqref{eq:yCHxy} then yields
\[
f(x,y)=y^\top c_H(x)y,
\]
where the matrix $H\in\S^{15}$ can be found to be
\[
H=\frac{1}{2}\left(\begin{array}{*{20}c}
2 & 0 & 0 & 0 & 0 & 0 & 0 & 0 & -2& -1 & 0 & 0 & 0  & 0 & 0\\
0 & 2 & 0 & 0 & 0 & 0 & -1 & 0 & 0 & 0 & -2 & 0 & 0  & 0& 0\\
0 & 0 & 4 & 0 & 0 & -1 & 0 & 0 & 0& 0 & 0 & 0 & 0  & 0& 0\\
0 & 0 & 0 & 4 & 0 & 0 & 0 & 0 & 0& 0 & 0 & 0 & 0  & 0& 0\\
0 & 0 & 0 & 0 & 4 & 0 & 0 & 0 & 0& 0 & 0 & 0 & 0  & 0& 0\\
0 & 0 & -1 & 0 & 0 & 4 & 0 & 0 & 0& 0 & 0 & 0 & -1  & 0& 0\\
0 & -1 & 0 & 0 & 0 & 0 & 2 & 0 & 0& 0 & -1 & 0 & 0  & 0& 0\\
0 & 0 & 0 & 0 & 0 & 0 & 0 & 4 & 0& -1 & 0 & 0 & 0  & 0& 0\\
-2 & 0 & 0 & 0 & 0 & 0 & 0 & 0 & 2& 0 & 0 & 0 & 0  & 0& 0\\
-1 & 0 & 0 & 0 & 0 & 0 & 0 & -1 & 0& 4 & 0 & 0 & 0  & 0& 0\\
0 & -2 & 0 & 0 & 0 & 0 & -1 & 0 & 0& 0 & 2 & 0 & 0  & 0& 0\\
0 & 0 & 0 & 0 & 0 & 0 & 0 & 0 & 0& 0 & 0 & 4 & 0  & 0& 0\\
0 & 0 & 0 & 0 & 0 & -1 & 0 & 0 & 0& 0 & 0 & 0 & 2  & 2& 0\\
0 & 0 & 0 & 0 & 0 & 0 & 0 & 0 & 0& 0 & 0 & 0 & 2  & 2& 0\\
0 & 0 & 0 & 0 & 0 & 0 & 0 & 0 & 0& 0 & 0 & 0 & 0  & 0& 4\\
\end{array}\right).
\]
The characteristic polynomial of $H$ is given by
\[
\det(s\,\Id - H)=2^{-5}(s-2)^7\,(2s^2 - 2s - 1)\,(4s^3 - 16s^2 + 14s + 1)^2.
\]
Here $2s^2 - 2s - 1$ has one single negative root $\lambda=(1-\sqrt{3})/2$. Moreover, the polynomial $\phi(s)=4s^3 - 16s^2 + 14s + 1$ has one single negative root $\mu$, since $\phi(0)=1>0$, $\phi(2)=-3<0$, $\lim_{s\to\infty}\phi(s)=\infty$, and $\lim_{s\to-\infty}\phi(s)=-\infty$. We now define the three vectors
\begin{equation*}
\begin{aligned}
v_1&=\frac{1}{2} e_2-\lambda \, e_7+\frac{1}{2} e_{11}\\
v_2&=\frac{\mu}{2} e_3+\mu(2-\mu)\, e_6+\frac{1}{2}(\mu-1) \, e_{13}+\frac{1}{2} e_{14}\\
v_3&=\frac{1}{2}(1-\mu)\, e_1-\frac{\mu}{2} \, e_8+\frac{1}{2} e_{9}+\mu(\mu-2)\, e_{10},
\end{aligned}
\end{equation*}
where $e_i$ denotes the $i$th canonical basis vector in $\R^{15}$. A calculation shows that $v_1$ is an eigenvector of $H$ corresponding to the eigenvalue $\lambda$, and $v_2$, $v_3$ are linearly independent eigenvectors of $H$ corresponding to the eigenvalue $\mu$. We now define the positive semidefinite matrix
\[
B = \delta v_1v_1^\top + v_2v_2^\top + v_3v_3^\top,
\]
where $\delta=\mu(\mu-2)(2\mu-1)/\lambda>0$. This matrix satisfies
\begin{equation}\label{eq:HandW}
\langle H, B\rangle = \tr(HB) = \lambda\delta\|v_1\|^2 + \mu \|v_2\|^2 + \mu\|v_3\|^2 <0.
\end{equation}
In addition, by considering each of the ${6\choose 4}=15$ basis elements $K_{(i,j,k,l)}$ of $\Kcal$ described in Remark~\ref{R:Kbasis}, one can check by direct computation that
\begin{equation}\label{eq:KandW}
\langle K, B\rangle=0 \quad \text{for all} \quad K\in\Kcal.
\end{equation}
Since $B$ is nonzero and positive semidefinite, \eqref{eq:HandW} and~\eqref{eq:KandW} imply \eqref{eq:HpKcSe}. This completes the proof of Theorem~\ref{T:sos}\ref{T:sos:2}.

\begin{remark}
The condition~\eqref{eq:HpKcSe} and its relation to sum of squares representations can also be expressed via the following primal-dual pair of semidefinite programs:
\begin{align}
v(P)=&\min_{B\in\S^m}\left\{ \langle H,B\rangle \colon B\succeq0,\ \langle\Id,B\rangle=1,\ \langle K,B\rangle=0\text{ for all }K\in\Kcal \right\} \label{eq:primalSDP}\\[2ex]
v(D)=&\max_{\substack{\lambda\in\R \\ K\in\Kcal}}\left\{ \lambda \colon H-\lambda\Id+K \succeq 0 \right\}.
\end{align}
Note that $v(D)\ge0$ if and only if $(H+\Kcal)\cap\S^{15}_+\ne\emptyset$. Moreover, weak duality $v(D)\le v(P)$ always holds. Indeed, for any primal feasible solution $B$ and dual feasible solution $(\lambda, K)$ one has $0\le \langle B,H-\lambda\Id+K\rangle = \langle H,B\rangle - \lambda$. Thus, for a given candidate $c=c_H$, a primal feasible solution $B$ with $\langle H,B\rangle<0$ provides a certificate that no sum of squares representation can exist. Furthermore, \eqref{eq:HandW} and~\eqref{eq:KandW} imply that the matrix $B/\tr(B)$ is thus a feasible solution to the primal problem~\eqref{eq:primalSDP} with negative value, and hence certifies that no sum of squares representation can exist.
\end{remark}

\begin{remark}
One can write down the component functions of the resulting map $c=c_H$:
\begin{align*}
c_{11}(x) &= x_2^2+x_3^2+2x_4^2+2x_5^2+2x_6^2			&c_{33}(x) &= x_1^2+2x_1x_5+2x_2^2+2x_4^2+x_5^2+2x_6^2\\
c_{12}(x) &= -x_1x_2-x_2x_6-x_3x_4					&c_{34}(x) &= -x_1x_2+x_2x_5-2x_3x_4\\
c_{13}(x) &= -x_1x_3-x_3x_5							&c_{35}(x) &= -x_1x_3-x_3x_5\\
c_{14}(x) &= x_2x_3-2x_1x_4							&c_{36}(x) &= -2x_3x_6\\
c_{15}(x) &= -2x_1x_5+x_3^2							&c_{44}(x) &= 2x_1^2+x_2^2+2x_3^2+x_5^2+2x_5x_6+x_6^2\\
c_{16}(x) &= -2x_1x_6+x_2^2							&c_{45}(x) &= -x_2x_3-x_4x_5-x_4x_6\\
c_{22}(x) &= x_1^2+2x_1x_6+2x_3^2+x_4^2+2x_5^2+x_6^2	&c_{46}(x) &= -x_4x_5-x_4x_6\\
c_{23}(x) &= x_1x_4-2x_2x_3-x_4x_5					&c_{55}(x) &= 2x_1^2+2x_2^2+x_3^2+x_4^2+2x_6^2\\
c_{24}(x) &= -x_2x_4								&c_{56}(x) &= x_4^2-2x_5x_6\\
c_{25}(x) &= x_3x_4-2x_2x_5							&c_{66}(x) &= 2x_1^2+x_2^2+2x_3^2+x_4^2+2x_5^2\\
c_{26}(x) &= -x_1x_2-x_2x_6
\end{align*}
\end{remark}

\begin{remark}
A natural question is whether $c=c_H$ is non-degenerate on the sphere, that is, if all eigenvalues are strictly positive except the one corresponding to the eigenvector $x$. This turns out not to be the case. With $x_2=x_4=x_6=0$ and under the restriction $\|x\|=1$, the eigenvalues of $c(x)$ are $0,(x_1+x_5)^2,2-x_1^2,2-x_5^2,2,2$. Thus $c(x)$ is degenerate along the curve $\Scal^{d-1}\cap\{x_2=x_4=x_6=0\}\cap\{x_1+x_5=0\}$. In particular, we do not know whether the corresponding diffusion on the sphere can be realized as the solution to an SDE for which pathwise uniqueness holds.
\end{remark}

\bibliographystyle{plainnat}
\bibliography{bibl}

\end{document}